\documentclass[a4paper,reqno]{amsart}
\usepackage[]{amsmath, amsthm, amsfonts, verbatim, amssymb}
\usepackage{tikz}
\usetikzlibrary{matrix,arrows}
\usepackage{url}
\usepackage[all,cmtip]{xy}

\usepackage{graphicx}
\usepackage{pinlabel}
\let\bbbibitem\bibitem
\renewcommand{\bibitem}[2][]{\bbbibitem[#1]{#2}\label{#2}}

\newcommand\C{{\mathbb C}}
\newcommand\projmap{{\varpi}}

\def\cM{{\mathcal M}}
\def\ocM{\overline{\mathcal M}}

\def\cH{\mathcal H}
\def\cO{{\mathcal O}}

\def\cT{{\mathcal T}}

\def\bZ{{\mathbb Z}}
\def\bQ{{\mathbb Q}}
\def\bC{{\mathbb C}}

\def\cA{{\mathcal A}}
\def\cB{{\mathcal B}}
\def\cT{{\mathcal T}}

\def\cM{{\mathcal M}}
\def\cP{{\mathcal P}}
\def\cC{{\mathcal C}}

\begin{document}
\newtheorem {theo}{Theorem}
\newtheorem {coro}{Corollary}
\newtheorem {conj}{Conjecture}
\newtheorem {lemm}{Lemma}
\newtheorem {rem}{Remark}
\newtheorem {defi}{Definition}
\newtheorem {ques}{Question}
\newtheorem {prop}{Proposition}
\def\spb{\smallpagebreak}
\def\mpb{\vskip 0.5truecm}
\def\bpb{\vskip 1truecm}
\def\wtM{\widetilde M}
\def\tM{\widetilde M}
\def\wtN{\widetilde N}
\def\tN{\widetilde N}
\def\tR{\widetilde R}
\def\tC{\widetilde C}
\def\tX{\widetilde X}
\def\tY{\widetilde Y}
\def\tE{\widetilde E}
\def\tL{\widetilde L}
\def\tQ{\widetilde Q}
\def\tS{\widetilde S}
\def\tx{\widetilde x}
\def\talpha{\widetilde\alpha}
\def\ti{\widetilde \iota}
\def\hM{\hat M}
\def\hq{\hat q}
\def\hR{\hat R}
\def\bs{\bigskip}
\def\ms{\medskip}
\def\ni{\noindent}
\def\td{\nabla}
\def\pd{\partial}
\def\hol{$\text{hol}\,$}
\def\Log{\mbox{Log}}
\def\bfQ{{\bf Q}}
\def\Todd{\mbox{Todd}}
\def\top{\mbox{top}}
\def\Pic{\mbox{Pic}}
\def\bP{{\bf P}}
\def\dxi{d x^i}
\def\dxj{d x^j}
\def\dyi{d y^i}
\def\dyj{d y^j}
\def\dzi{d z^I}
\def\dzj{d z^J}
\def\ozi{d{\overline z}^I}
\def\ozj{d{\overline z}^J}
\def\oz1{d{\overline z}^1}
\def\oz2{d{\overline z}^2}
\def\oz3{d{\overline z}^3}
\def\sI{\sqrt{-1}}
\def\hol{$\text{hol}\,$}
\def\ok{\overline k}
\def\ol{\overline l}
\def\oJ{\overline J}
\def\oT{\overline T}
\def\oS{\overline S}
\def\oV{\overline V}
\def\oW{\overline W}
\def\oY{\overline Y}
\def\oL{\overline L}
\def\oI{\overline I}
\def\oK{\overline K}
\def\oE{\overline E}
\def\oL{\overline L}
\def\oj{\overline j}
\def\oi{\overline i}
\def\ok{\overline k}
\def\oz{\overline z}
\def\om{\overline mu}
\def\on{\overline nu}
\def\oa{\overline \alpha}
\def\ob{\overline \beta}
\def\oGamma{\bar \Gamma}
\def\of{\overline f}
\def\og{\overline \gamma}
\def\ogamma{\overline \gamma}
\def\odelta{\overline \delta}
\def\otheta{\overline \theta}
\def\ophi{\overline \phi}
\def\opd{\overline \partial}
\def\oA{\overline A} 
\def\oB{\overline B}
\def\oC{\overline C}
\def\op{\overline D}
\def\ov{\overline \varphi}
\def\oIq1{\oI_1\cdots\oI_{q-1}}
\def\oIq2{\oI_1\cdots\oI_{q-2}}
\def\op{\overline \partial}
\def\ua{{\underline {a}}}
\def\us{{\underline {\sigma}}}
\def\Chow{{\mbox{Chow}}}
\def\vol{{\mbox{vol}}}
\def\rank{{\mbox{rank}}}
\def\diag{{\mbox{diag}}}
\def\pr{\mbox{pr}}
\def\tor{\mbox{tor}}
\def\bp{{\bf p}}
\def\bk{{\bf k}}
\def\a{{\alpha}}
\def\tchi{\widetilde{\chi}}
\def\ta{\widetilde{\alpha}}
\def\ovarphi{\overline \varphi}
\def\ocH{\overline{\cH}}
\def\tV{\widetilde{V}}
\def\tf{\widetilde{f}}
\def\th{\widetilde{h}}
\def\tT{\widetilde{T}}
\def\ocM{\overline{\cM}}
\def\ocA{\overline{\cA}}
\def\ocB{\overline{\cB}}
\def\o{\mbox{orb}}
\def\Jac{\mbox{Jac}}
\def\Aut{{\rm Aut}}
\def\tC{\widetilde{C}}
\def\oQ{\overline Q}
\def\tpi{\widetilde\pi}
\def\hP{\widehat P}
\def\hC{\widehat C}
\def\hD{\widehat D}
\def\obQ{\overline{\bQ}}
\def\IIm{\mbox{Im}}
\def\hE{\widehat{E}}
\def\tdelta{\widetilde{\delta}}
\def\tD{\widetilde{D}}
\def\wh{\widehat}
\def\hX{\widehat X}
\def\hq{\widehat q}
\def\hC{\widehat C}
\def\oX{\overline X}
\def\P{{\mathbb P}}
\def\Q{{\mathbb Q}}
\def\Z{{\mathbb Z}}
\def\PU{{\rm PU}}
\def\U{{\rm U}}
\def\SU{{\rm SU}}
\def\GL{{\rm GL}}
\def\PGL{{\rm PGL}}
\def\F{{\mathbb F}}
\def\Ro{{\hat R_{N,d,2}}}
\def\Po{{\P^1_{2,3,N}}}
\newcommand{\Ham}{{\mathbb H}}
\newcommand{\twoball}{B^2_\C}
\newcommand{\Adj}{\mathrm{Adj}}
\newcommand{\XX}{{X}}
\newcommand{\cd}{{c'}}
\newcommand{\comp}{{{\scriptstyle\circ}}}
\def\widebar{\overline}
\def\Dred{D^{{\rm red}}}
\def\dan{\delta^{{\rm an}}}
\def\dtop{\delta^{{\rm top}}}
\def\cru{\omega}
\def\NL{\Delta}
\def\unk{\kappa}
\def\pro{q}

\setcounter{section}{-1}

\author[Donald I.~Cartwright, Vincent Koziarz, Sai-Kee Yeung]{Donald I.~Cartwright, Vincent Koziarz, Sai-Kee Yeung}
\address[Donald I.~Cartwright]{University of Sydney, Sydney NSW 2006 Australia}
\email{donald.cartwright@sydney.edu.au}
\address[Vincent Koziarz]{Univ. Bordeaux, IMB, UMR 5251, F-33400 Talence, France}
\email{vkoziarz@math.u-bordeaux1.fr}
\address[Sai-Kee Yeung]{Purdue University, West Lafayette, IN 47907, USA}
\email{yeung@math.purdue.edu}

\title
[On the Cartwright-Steger surface]
{On the Cartwright-Steger surface}

{}

\thanks{ \\Key words: Special complex surface, Albanese fibration, Complex ball quotients\\
{\it AMS 2010 Mathematics subject classification: Primary 14D06, 14J29, 32J15} \\
The third author was partially supported by a grant from the National Science Foundation}

\maketitle
\ni{\bf Abstract} {\it  In this article, we study various concrete algebraic and differential geometric properties of the Cartwright-Steger surface,
the unique smooth surface of Euler number~3 which is neither a projective plane nor a fake projective plane.
In particular, we determine the genus of a generic fiber of the Albanese fibration, and deduce that the singular fibers are not totally geodesic, 
answering an open problem about fibrations of a complex ball quotient over a Riemann surface.
} 

\section{Introduction}

The Cartwright-Steger surface was found during work on the
classification of fake projective planes completed in \cite{PY} and \cite{CS}.
A fake projective plane is a smooth surface with the same Betti
numbers as the projective plane but not biholomorphic to it.  It is
known that a fake projective plane is a complex two ball quotient
$\Pi\backslash B_{\bC}^2$ with Euler number~3, where $\Pi$ is an
arithmetic lattice in $\PU(2,1)$, cf. \cite{PY}.  In the scheme of
classification of fake projective planes carried out in \cite{PY}, it was
conjectured but not proved in~\cite{PY} that the lattice~$\Pi$ associated
to a fake projective plane cannot be defined over a pair of number
fields $\cC_{11}=(\bQ(\sqrt3),\bQ(\zeta_{12}))$, where $\zeta_{12}$ is
a $12$-th root of unity. Such a~$\Pi$ would be of index~864 in a
certain maximal arithmetic subgroup~$\bar\Gamma$ of~$\PU(2,1)$. As
reported in~\cite{CS}, the authors showed using a lengthy computer search
that there is no torsion free lattice $\Pi$ of index~864
in this~$\bar\Gamma$ with $b_1(\Pi)=0$, but surprisingly there is one with
$b_1(\Pi)=2$.  The surface $\Pi\backslash
B_{\bC}^2$ is the subject of study in this article.

The Cartwright-Steger surface is unique as a Riemannian manifold 
with the given Euler and first Betti numbers, but
has two different biholomorphic structures given by complex
conjugation.  From an algebraic geometric point of view,
the fake projective planes and the
Cartwright-Steger surfaces are interesting since they have the
smallest possible Euler number, namely~3, among smooth surfaces of
general type, and constitute all such surfaces.  From a differential geometric point of view,
they are interesting since they constitute smooth complex hyperbolic space forms, or complex ball quotients, of smallest
volume in complex dimension two.
We refer the reader to \cite{R}, \cite{Y1}, and \cite{Y} for some general discussions related
to the above facts.  Unlike fake projective 
planes, whose lattices arise from division algebras
of non-trivial degree as classified, the Cartwright-Steger surface is
defined by Hermitian forms over the number fields mentioned above.  It is
realized among experts that such a surface is commensurable to a
Deligne-Mostow surface, the type of surfaces which have been studied
by Picard, Le Vavasseur, Mostow, Deligne-Mostow, Terada and many
others, cf.~\cite{DM1}.

Even though the lattice involved is described in \cite{CS2}, it is
surprising that the algebraic geometric structures of the surface are
far from being understood.  A typical problem is to find out the genus
of a generic fiber of the associated Albanese fibration.  Conventional
algebraic geometric techniques do not seem to be readily applicable to
such a problem.  The goal of this article is to develop tools and
techniques which allow us to understand concrete surfaces such as the
Cartwright-Steger surface.  In particular, we recover algebraic
geometric properties from a description of the fundamental group of the
surface, using a combination of various algebraic geometric,
differential geometric, group theoretical techniques and computer
implementations.

Here are the results obtained in this paper.  

\bs
\ni{\bf Main Theorem} 
{\it Let $X$ be the Cartwright-Steger surface and $\a: X\rightarrow T$ the Albanese map.
\begin{itemize}
\item[(a)] The genus of a generic fiber of $\a$ is $19$.  
\item[(b)] All fibers of $\a$ are reduced.
\item[(c)] The Albanese torus $T$ is $\bC/(\bZ+\cru\bZ)$, where $\cru$ is a cube root of unity.
\item[(d)] The Picard number of $X$ is $3$, equal to $h^{1,1}(X)$, so that all the Hodge $(1,1)$ classes are algebraic. 
The N\'eron-Severi group is generated by three immersed totally geodesic curves we explicitly give.
\item[(e)] The automorphism group~$\Sigma$ of $X$, isomorphic to $\bZ_3$, has $9$ fixed points, and induces
a nontrivial action on~$T$ which has 3 fixed points. Three fixed points of~$\Sigma$ lie over each fixed point in~$T$.
Over one fixed point on~$T$, the three fixed points of~$\Sigma$ are of type~$\frac13{(1,1)}$. The
other~6 fixed points of~$\Sigma$ are of type $\frac13{(1,2)}$.
\end{itemize}
}
\bs

The Main Theorem follows from Theorem~\ref{th:genus}, Corollary~\ref{coro:numsing}, Lemma~\ref{lem:actiontorus}, Corollary~\ref{coro:picard}, Lemma~\ref{lem:numind} and Proposition~\ref{prop:fixpoints}.
As an immediate consequence, see Theorem~\ref{thm:answertomok}, we have answered an open problem communicated 
to us by Ngaiming Mok on properties of fibrations on complex ball quotients.

\ms
\ni{\bf Corollary} 
{\it There exists a fibration of a smooth complex two ball quotient over a smooth Riemann surface with non-totally geodesic singular fibers.}

\ms


 Here are a few words about the presentation of the article. To streamline our arguments
and to make the results more understandable, we state and prove the geometric
results of the article sequentially in the sections 3 to 5 of the article. Many of the results rely on computations in the groups $\Pi$ 
and~$\bar\Gamma$, often obtained with assistance of the algebra package Magma, and we present these exclusively in the first two sections of the paper (except for the proof of Proposition~\ref{prop:fixpoints}) with a geometric perspective each time it is possible. More details appear in a longer version of this paper and on the webpage of the first author~\cite{CKY}.

\medskip
\noindent{\it Acknowledgments.} Donald Cartwright thanks Jonathan Hillman for help in obtaining
standard presentations of surface groups from non-standard ones.
Vincent Koziarz would like to thank
Aurel Page for his help in computations involving Magma, Duc-Manh
Nguyen for very useful conversations about triangle groups and Riemann
surfaces, Arnaud Ch\'eritat for his help in drawing pictures and Fr\'ed\'eric Campana for useful comments. Sai-Kee Yeung would 
like to thank Martin Deraux, Igor Dolgachev, Ching-Jui Lai, Ngaiming Mok and
Domingo Toledo for their interest and helpful comments. 
The main results of this paper were presented at the 4th South 
Kyushu workshop on algebra, complex ball quotients and related topics, 
July 22--25, 2014, Kumamoto, Japan.  The first and the third authors thank Fumiharu Kato for his kind invitation.

Since completing this paper, we were informed by Domingo Toledo that he, 
Fabrizio Catanese, JongHae Keum and Matthew Stover had independently proved 
some of our results in a paper they are preparing.

\section{Basic facts}\label{sec:basic}

\subsection{}\label{sec:matrix}
Let $F$ be a Hermitian form on $\C^3$ with signature $(2,1)$. We denote by $\U(F)=\{g\in\GL(3,\C)\,|\,g^*Fg=F\}$ the subgroup of $\GL(3,\C)$ preserving the form $F$, by $\SU(F)$ the subgroup of $\U(F)$ of elements with determinant 1, and by $\PU(F)$ their image in $\PGL(3,\C)$. The group $\PU(F)$ is naturally identified with the group of biholomorphisms of the two-ball $\twoball(F):=\{[z]\in\P^2_\C=\P(\C^3)\,|\,F(z)<0\}$.

Our aim is to study a special complex hyperbolic surface $X=\Pi\backslash\twoball(F)$ where $\Pi$ is a cocompact torsion-free lattice in some $\PU(F)$. The group $\Pi$ appears as a finite index subgroup of an arithmetic lattice $\bar\Gamma$ which can be easily described as follows.

Let $\zeta=\zeta_{12}$ be a primitive 12-th root of unity. Then
$r=\zeta+\zeta^{-1}$ is a square root of~3. Let $\ell=\Q(\zeta)$
and $k=\Q(r)\subset\ell$. For 
real and complex calculations below, we take $\zeta=e^{\pi i/6}$, and then $r$ is
the positive square root of~3. We could define $\bar\Gamma$ to be the group of 
$3\times3$ matrices~$g'$ with entries in~$\ell$ such that ${g'}^*F'g'=F'$, where
\begin{displaymath}
F'=\begin{pmatrix}
r+1&-1&0\\
-1&r-1&0\\
0&0&-1
\end{pmatrix},
\end{displaymath}
such that $g'$ has entries in~$\Z[\zeta]$, modulo $Z=\{\zeta^jI:j=0,\ldots,11\}$. 

However, it is convenient to work with a diagonal form instead of~$F'$. Notice that 
$F'=(r-1)^{-1}\gamma_0^*F\gamma_0$ for
\begin{displaymath}
F=\begin{pmatrix}
1&0&0\\
0&1&0\\
0&0&1-r
\end{pmatrix},
\quad\text{and}\quad
\gamma_0=\begin{pmatrix}
1&  0&0\\
1&1-r&0\\
0&  0&1
\end{pmatrix}.
\end{displaymath}
So we instead define $\bar\Gamma$ to be the group of matrices~$g$, modulo~$Z$, with entries 
in~$\ell$, which are unitary with respect to~$F$ for which $g'=\gamma_0^{-1}g\gamma_0$
has entries in~$\Z[\zeta]$. Such $g$'s have entries in $\frac{1}{r-1}\Z[\zeta]\subset\frac{1}{2}\Z[\zeta]$.

Since $F$ is diagonal, it is easy to make the group $\PU(F)$ act on the standard unit two-ball, which we will just denote by $\twoball$:  if $gZ\in\bar\Gamma$, the action of $gZ$ on~$\twoball$ is given by 
\begin{displaymath}
(gZ).(z,w)=(z',w')\quad\text{if}\quad 
DgD^{-1}\begin{pmatrix}z\\w\\1\end{pmatrix}=\lambda\begin{pmatrix}z'\\w'\\1\end{pmatrix},
\end{displaymath}
for some $\lambda\in\C$, where $D$ is the diagonal matrix with diagonal entries 1, 1 and~$\sqrt{r-1}$. 
We will ignore the distinction between matrices $g$ and elements $gZ$ of~$\bar\Gamma$.

Now $\bar\Gamma$ contains a subgroup~$K$ of order 288 generated
by the two matrices $u=\gamma_0u'\gamma_0^{-1}$ and $v=\gamma_0v'\gamma_0^{-1}$
where
\begin{displaymath}
u'=\begin{pmatrix}
\zeta^3+\zeta^2-\zeta&1-\zeta&0\\
\zeta^3+\zeta^2-1&\zeta-\zeta^3&0\\
0&0&1
\end{pmatrix}
\quad\text{and}\quad
v'=\begin{pmatrix}
\zeta^3&0&0\\
\zeta^3+\zeta^2-\zeta-1&1&0\\
0&0&1\\
\end{pmatrix}.
\end{displaymath}
A presentation for~$K$ is given by the relations
$u^3=v^4=1$ and $(uv)^2=(vu)^2$.
The elements of~$K$ are most neatly expressed if we use not only the generators $u$ and~$v$,
but also $j=(uv)^2$, which is the diagonal matrix with diagonal entries $\zeta$, $\zeta$ and~1, 
and generates the center of~$K$.

There is one further generator needed for~$\bar\Gamma$, namely $b=\gamma_0b'\gamma_0^{-1}$
for
\begin{displaymath}
b'=\begin{pmatrix}
                         1&                     0&0\\
-2\zeta^3-\zeta^2+2\zeta+2&\zeta^3+\zeta^2-\zeta-1&-\zeta^3-\zeta^2\\
             \zeta^2+\zeta&             -\zeta^3-1&-\zeta^3+\zeta+1
\end{pmatrix}.
\end{displaymath}

\begin{theo}[\cite{CS2}]\label{thm:barGammapresentation}A presentation of~$\bar\Gamma$ is given by the generators
$u$, $v$ and~$b$ and the relations
\begin{displaymath}
u^3=v^4=b^3=1,\ (uv)^2=(vu)^2,\ v b=bv,\ (buv)^3=(buvu)^2v=1.
\end{displaymath}
\end{theo}

\subsection{} Let us record here the connection with a group which was first discovered by Mostow: the group $\bar\Gamma$ is isomorphic to a group generated by complex reflections, denoted by $\Gamma_{3,\frac13}$ in the paper~\cite{Mos1} and by $\Gamma_{3,4}$ in~\cite{Parker}, and whose presentation (see Parker~\cite{Parker}) is
\begin{displaymath}
\Gamma_{3,4}=\langle J,R_1,A_1\ :\ J^3=R_1^3=A_1^4=1,\ A_1=(JR_1^{-1}J)^2,\ A_1R_1=R_1A_1\rangle.
\end{displaymath}
Defining $R_2=JR_1J^{-1}$, it was shown in~\cite[Proposition 4.6]{Parker} that the subgroup
$\langle A_1,R_2\rangle$ of~$\Gamma_{3,4}$ is finite, with order~288 (actually, it is isomorphic to $K$ above). It has the simple presentation
\begin{displaymath}
\langle A_1,R_2\ :\ A_1^4=R_2^3=1,\ A_1R_2A_1R_2=R_2A_1R_2A_1\rangle.
\end{displaymath}
The following result was communicated to us by~John Parker.

\begin{prop}\label{prop:DMiso}There is an isomorphism $\psi:\bar\Gamma\to\Gamma_{3,4}$ such that
\begin{displaymath}
\psi(u)=R_2,\ \psi(v)=A_1,\ \text{and}\ \psi(b)=R_1.
\end{displaymath}
It satisfies $\psi(K)=\langle A_1,R_2\rangle$, and its inverse satisfies
\begin{displaymath}
\psi^{-1}(R_1)=b,\ \psi^{-1}(A_1)=v,\ \psi^{-1}(J)=buv,\quad\text{and}\ \ \psi^{-1}(R_2)=u.
\end{displaymath}
\end{prop}

\subsection{}\label{sec:descorb} It is also convenient to see $\bar\Gamma$ as a (Deligne-)Mostow group: it corresponds to item 8 in the paper of Mostow \cite[p. 102]{Mos2} whose associated weights $(2,2,2,7,11)/12$ satisfy the condition ($\Sigma$INT) in the notation of~\cite{Mos2}. We refer to ~\cite{Mos2} and \cite{DM2} for details on the description below.

The orbifold quotient $\bar\Gamma\backslash\twoball$ is a compactification of the moduli space of $5$-tuples of distinct points $(x_0,x_1,x_2,x_3,x_4)\in(\P^1_\C)^5$ modulo the diagonal action of $\PGL(2,\C)$ and the action of the symmetric group on three letters $\Sigma_3$ on the three first points. The compactification can be described as follows. First, it can be easily seen that the moduli space $Q$ of $5$-tuples of distinct points $(x_0,x_1,x_2,x_3,x_4)\in(\P^1_\C)^5$ modulo the diagonal action of $\PGL(2,\C)$ can be realized as $\P^2_\C$ with a configuration of six lines removed. In homogeneous coordinates $[X_0:X_1:X_2]$ on $\P^2_\C$, these six lines correspond to the three lines of ``type $A$'' with equation $X_i=X_j$ ($1\leqslant i<j\leqslant 2$) and the three lines of ``type $B$'' with equation $X_i=0$ ($i=0,1,2$).
In fact, the compactification $\bar Q=\P^2_\C$ of $Q$ is determined by the fact that we allow two or three of the points $x_0$, $x_1$ and~$x_2$ to coincide ($x_0=x_1$ corresponds to $X_0=X_1$, $x_0=x_2$ to $X_0=X_2$ and $x_1=x_2$ to $X_1=X_2$) and we also allow one or two of the points $x_0$, $x_1$ and~$x_2$ to coincide with $x_3$ ($x_0=x_3$ corresponds to $X_0=0$, $x_1=x_3$ to $X_1=0$ and $x_2=x_3$ to $X_2=0$).

Then, as we mentioned above, the underlying topological space of $\bar\Gamma\backslash\twoball$ is a compactification $R$ of $Q/\Sigma_3$ and actually is the weighted projective plane $\P(1,2,3)\cong\P^2_\C/\Sigma_3$ where the symmetric group on three letters $\Sigma_3$ acts by permutation of the homogeneous coordinates $[X_0:X_1:X_2]$ on $\P^2_\C$. There are two remarkable (irreducible) divisors on $\P(1,2,3)$: one is the image $D_A$ of the divisors of type $A$, the other one is the image $D_B$ of the divisors of type $B$. The divisor $D_A$ has a cusp at the image $P_1$ of the point $[1:1:1]$ and the divisor $D_B$ is smooth. These two divisors meet at two points: once at the image $P_2$ of the points $[1:0:0]$, $[0:1:0]$ or $[0:0:1]$ where they are tangent, once at the image $P_3$ of the points $[1:1:0]$, $[1:0:1]$ or $[0:1:1]$ where the intersection is transverse. 
There are also two singular points on $\P(1,2,3)$: one is a singularity of type $A_1$ and is the image $P_4\in D_B$ of the points $[1:-1:0]$, $[1:0:-1]$ or $[0:1:-1]$, the other one is a singularity of type $A_2$ and is the image $P_5$ of the points $[1:\omega:\omega^2]$ or $[1:\omega^2:\omega]$ where $\omega$ is a primitive 3rd root of unity.

\begin{figure}
\labellist
\small
\pinlabel $P_1$ at 287 125
\pinlabel $P_2$ at 377 132
\pinlabel $P_3$ at 345 73
\pinlabel $P_4$ at 325 35
\pinlabel $P_5$ at 285 60
\pinlabel $D_A$ at 360 200
\pinlabel $D_B$ at 410 185
\endlabellist
\includegraphics[scale=0.7]{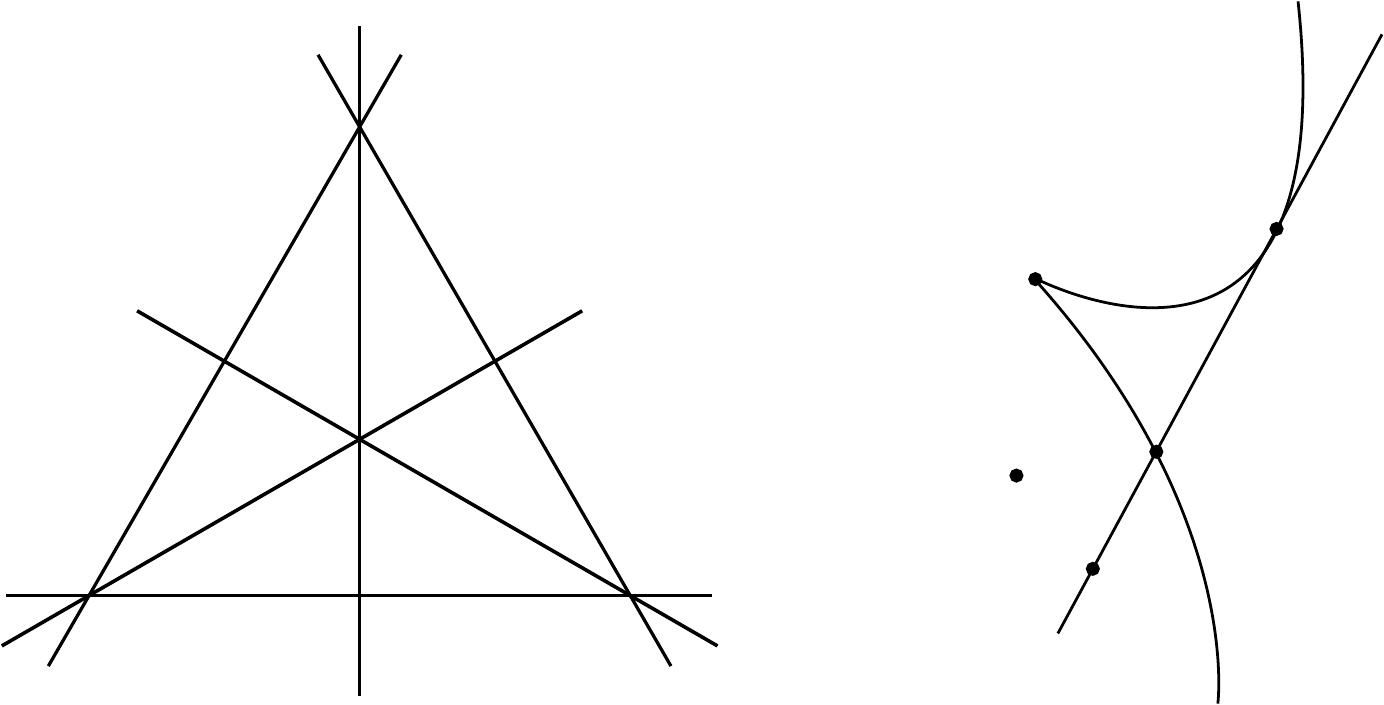}
\caption{$\bar Q=\P^2_\C$ and $R=\P^2_\C/\Sigma_3$}\label{fig:barQR}
\end{figure}

\begin{rem}
In the book~\cite[p. 111]{DM2}, the divisor $D_A$ (resp. $D_B$) is denoted by $D_{AA}$ (resp. $D_{AB}$) and the points $P_1,\dots,P_5$ simply by $1,\dots,5$.
\end{rem}

There is a standard method to compute the weight of the orbifold
divisors on $\bar\Gamma\backslash\twoball$ as well as the local groups
at the orbifold points, according to the weights
$(2,2,2,7,11)/12$. The weight of $D_A$ is $3=2(1-(2+2)/12)^{-1}$ and
the weight of $D_B$ is $4=(1-(2+7)/12)^{-1}$. This means that the
preimage of $D_A$ (resp. $D_B$) in $\twoball$ is a union of mirrors of
complex reflections of order 3 (resp. 4). We will denote by $\cM_A$
(resp. $\cM_B$) the corresponding sets of mirrors and we will refer to
these sets as {\it mirrors\/} of types~$A$ (resp. $B$). Said another
way, the isotropy group at a generic point of some $M\in \cM_A$ is
isomorphic to $\Z_3$ and the isotropy group at a generic point of some
$M\in \cM_B$ is isomorphic to $\Z_4$, both generated by a complex
reflection of the right order. This has to be compared with the
description of $\bar\Gamma$ as~$\Gamma_{3,4}$.

The isotropy group at a point above the transverse intersection $P_3$
of $D_A$ and $D_B$ is naturally isomorphic to $\Z_3\times\Z_4$. As
$P_5$ is a singularity of type $A_2$ but does not belong to any
orbifold divisor, the local group at $P_5$ is isomorphic to
$\Z_3$. But since $P_4\in D_A$ is a singularity of type $A_1$, the
local group at $P_4$ has order $8=2\cdot 4$ and actually is isomorphic
to $\Z_8$ as can be seen using Magma.

It is a little bit more difficult to determine the isotropy group above the points $P_1$ and~$P_2$. It will also be useful to describe the stabilizer in $\oGamma$ of a mirror. For this, one can use a method similar to the one in \cite[Lemma 2.12]{Der1} and obtain the following lemma which already appeared in an unpublished manuscript of Deraux and Yeung.

\begin{lemm}\label{lem:isotropystab}
Let $\cM_A$ (resp. $\cM_B$) denote the set of mirrors of complex
reflections of order~3 (resp. 4) in $\oGamma$. 
  
Let $\cP\subset\twoball$ denote the set of points above $P_1$
and $\cT\subset\twoball$ denote the set of points above~$P_2$.
The following holds.
  \begin{itemize}
\item[(a)]The group $\oGamma$ acts transitively on $\cM_A$, on $\cM_B$,
    on $\cP$ and on $\cT$.
\item[(b)]  For each point $\xi\in\cP$, the
    stabilizer of $\xi$ is the one labelled $\sharp 4$ in the Shephard-Todd list. It is a central extension of a $(2,3,3)$-triangle
    group, with center of order $2$, and has order $24$. There are
    precisely $4$ mirrors in $\cM_A$ through each such $\xi\in\cP$.
  \item[(c)] For each point $\xi\in\cT$, the
    stabilizer of $\xi$ is the one labelled $\sharp 10$ in the Shephard-Todd list. It is a central extension of a $(2,3,4)$-triangle
    group, with center of order $12$, and has order $288$. Through
    each such $\xi\in\cT$, there are $8$ elements of $\cM_A$ and $6$
    elements of $\cM_B$.
  \item[(d)] The stabilizer of any element $M\in\cM_A$ is
    a central extension of a $(2,4,12)$-triangle group, with center of
    order $3$.
  \item[(e)] The stabilizer of any element $M\in\cM_B$ is
    a central extension of a $(2,3,12)$-triangle group, with center of
    order $4$.
    \end{itemize}
\end{lemm}

\ni{\it Sketch of proof.} (a) Follows from the above discussion.

(b) The point $P_1$ corresponds to $x_0=x_1=x_2$ so that the computation $3/2=(1-(2+2)/12)^{-1}$ shows that the spherical triangle group associated to the projective action of the isotropy group at $\xi\in\cP$ is $(2,3,3)$. Indeed, we have to consider the triangle with angles $(2\pi/3,2\pi/3,2\pi/3)$ and take the symmetry into account (i.e. dividing the triangle into six parts), so that we obtain a triangle with angles $(\pi/2,\pi/3,\pi/3)$. The center has order given by $2=(1-(2+2+2)/12)^{-1}$. Comparing with \cite[Table 1]{ST}, we see that the relevant group is the one labelled $\sharp 4$ in the Shephard-Todd list and the rest of the assertion follows.

(c) Similarly, the point $P_2$ corresponds for instance to $x_0=x_1=x_3$ and the additional computation $4=(1-(2+7)/12)^{-1}$ shows that the spherical triangle group associated to the projective action of the isotropy group at $\xi\in\cT$ is $(2,3,4)$. Indeed, we have to consider the triangle with angles $(\pi/4,\pi/4,2\pi/3)$ and take the symmetry into account (i.e. dividing the triangle into two parts), so that we obtain a triangle with angles $(\pi/2,\pi/3,\pi/4)$. The center has order given by $12=(1-(2+2+7)/12)^{-1}$. Comparing with \cite[Table 2]{ST}, we see that the relevant group is the one labelled $\sharp 10$ in the Shephard-Todd list.

(d) Follows from the interpretation of the stabilizer of $M\in\cM_A$ as a central extension with center of order 3 (corresponding to the order of the reflection with mirror $M$) of a Deligne-Mostow group with weights $(2,4,7,11)/12$ coming for instance from the collapsing of $x_1$ and $x_2$. The associated triangle group is $(2,4,12)$ since $2=(1-(2+4)/12)^{-1}$, $4=(1-(2+7)/12)^{-1}$ and $12=(1-(4+7)/12)^{-1}$.

(e) Similarly, the stabilizer of $M\in\cM_B$ is a central extension with center of order 4 (corresponding to the order of the reflection with mirror $M$) of a (Deligne-)Mostow group with weights $(2,2,9,11)/12$ coming for instance from the collapsing of $x_2$ and $x_3$. We have moreover to take care of the symmetry coming from the first two weights. The associated triangle group is $(2,3,12)$ since $3/2=(1-(2+2)/12)^{-1}$ and $12=(1-(2+9)/12)^{-1}$ so that we have to divide into two parts a triangle with angles $(2\pi/3,\pi/12,\pi/12)$.
\qed

\subsection{} We come back to the description of $\twoball$ and $\bar\Gamma$ in the more concrete terms of \S\ref{sec:matrix}.
The elements $u$ and~$v$ of~$\bar\Gamma$ are complex reflections of order~3 and~4, respectively. 
For $\alpha\in\C$, define
\begin{displaymath}
M_\alpha=\{(z,w)\in\twoball:z=\alpha w\}.
\end{displaymath}
We also let $M_\infty=\{(z,w)\in\twoball:w=0\}$. Setting $c=(r-1)(\zeta^3-1)/2=\zeta^2-\zeta$,
one can check that $u$ fixes each point of~$M_c$, and $v$ fixes each point of~$M_0$.
As a consequence of Lemma~\ref{lem:isotropystab}(a), $\cM_A=\{g(M_c):g\in\bar\Gamma\}$ and $\cM_B=\{g(M_0):g\in\bar\Gamma\}$. Of course, $g(M_c)$ and~$g(M_0)$ are the
sets of points of~$\twoball$ fixed by the complex reflection $gu g^{-1}$, and~$gv g^{-1}$, respectively. For $\xi\in\twoball$, let $\cM_A(\xi)$, respectively $\cM_B(\xi)$ denote the set of distinct mirrors~$M$,
of type~$A$ and~$B$, respectively, containing~$\xi$. 

\begin{prop}\label{prop:torsionelts}The non-trivial elements of finite order in~$\bar\Gamma$ are all conjugate
to one of the elements in the following table, or the inverse of one of these.
\begin{center}
\begin{tabular}[t]{|c|c|}\hline

\vbox to 2.5ex{}$d$&Representatives of elements of order $d$\\[0.25ex]\hline
\vbox to 2.5ex{}$2$&$v^2$, $j^6,\ (bu^{-1})^2$\\[0.25ex]\hline  
\vbox to 2.5ex{}$3$&$u$, $j^4$, $u j^4$, $buv$\\[0.25ex]\hline  
\vbox to 2.5ex{}$4$&$v$, $j^3$, $v j^3$, $v^2j^3$, $bu^{-1}$\\[0.25ex]\hline 
\vbox to 2.5ex{}$6$&$j^2$, $v^2j^2$, $v^2u j$, $v^2u j^5$, $bv^2u^{-1}j$, $bv^2$\\[0.25ex]\hline
\vbox to 2.5ex{}$8$&$uv j$, $\zeta^{-1}bj$, $(\zeta^{-1}bj)^3$\\[0.25ex]\hline
\vbox to 2.5ex{}$12$&$j$, $j^5$, $uv^{-1}j^2$, $uv^{-1}j^3$, $uv^{-1}j^6$, $uv^{-1}j^{-1}$, $v^2j$, $uv^2$, $u j$, $u j^3$, $bv$, $(bv)^{-5}$\\[0.25ex]\hline
\vbox to 2.5ex{}$24$&$uv$, $vu j^2$\\[0.25ex]\hline
\end{tabular}
\end{center}
\end{prop}
\begin{proof}
Elements of $\bar\Gamma$ which fix points of~$\twoball$ must have finite order, 
because $\bar\Gamma$ acts discontinuously on~$\twoball$.
Conversely (see \cite[Lemma~3.3]{CS2}) any element of finite order in~$\bar\Gamma$ fixes at least one point of~$\twoball$,
and is conjugate to an element of~$K\cup bK\cup bu^{-1}bK$.
One can easily list the nontrivial elements of finite order in this last set (there 
are 408 of them, 76 in~$bK$ and 45 in~$bu^{-1}bK$), all having order dividing~24. 
Routine calculations show that any such element (and hence
each nontrivial element of finite order in $\bar\Gamma$) has a matrix
representative~$g$ conjugate to one of the elements in the above
table, or its inverse.
\end{proof}

For $\alpha\in\C\cup\{\infty\}$ and for $\xi\in\twoball$, let
\begin{displaymath}
\bar\Gamma_\alpha=\{g\in\bar\Gamma:g(M_\alpha)=M_\alpha\}\quad\text{and}\quad\bar\Gamma_\xi=\{g\in\bar\Gamma:g.\xi=\xi\}
\end{displaymath}
denote the stabilizer of~$M_\alpha$ and~$\xi$, respectively. In \S\ref{sec:descorb}, we described the $\xi\in\twoball$ for which $\bar\Gamma_\xi\ne\{1\}$. The result can be summed up as follows:
\begin{center}
\begin{tabular}[t]{|c|c|c|c|}\hline

\vbox to 2.5ex{}$\bar\Gamma(\xi)$&$|\bar\Gamma_\xi|$&$|\cM_A(\xi)|$&$|\cM_B(\xi)|$\\[0.25ex]\hline
\vbox to 2.5ex{}$P_1$&$24$&$4$&$0$\\[0.25ex]\hline
\vbox to 2.5ex{}$P_2$&$288$&$8$&$6$\\[0.25ex]\hline
\vbox to 2.5ex{}$P_3$&$12$&$1$&$1$\\[0.25ex]\hline
\vbox to 2.5ex{}$P_4$&$8$&$0$&$1$\\[0.25ex]\hline
\vbox to 2.5ex{}$P_5$&$3$&$0$&$0$\\[0.25ex]\hline
\vbox to 2.5ex{}\hbox{\rm generic} $D_A$&$3$&$1$&$0$\\[0.25ex]\hline
\vbox to 2.5ex{}\hbox{\rm generic} $D_B$&$4$&$0$&$1$\\[0.25ex]\hline
\end{tabular}
\end{center}
where {\it generic} $D_A$ (resp. $D_B$) means that $\bar\Gamma(\xi)\in D_A$ (resp. $D_B$) and $\bar\Gamma(\xi)\ne P_1,P_2,P_3$ (resp. $P_2,P_3,P_4$). Two points of $\twoball$ are particularly important: the origin~$O$, such that $\bar\Gamma(O)=P_2$ (i.e. $O\in\cT$), and 
\begin{equation}\label{eq:Pdefn}
P=\Bigl(\frac{c(\zeta-1)}{\sqrt{r-1}},\frac{\zeta-1}{\sqrt{r-1}}\Bigr),
\end{equation}
such that $\bar\Gamma(P)=P_1$ (i.e. $P\in\cP$). In fact, $\bar\Gamma_O=K$ and, as the table above shows, $\bar\Gamma_P$ has cardinality~24 . Another important point will be the fixed point 
\begin{equation}\label{eq:xi3xi8xi12}
Q=\bigl(\frac{c_1}{\sqrt{r-1}},\frac{c_2}{\sqrt{r-1}}\bigr)
\end{equation}
of $buv$ such that $\bar\Gamma(Q)=P_5$ where for $\lambda=e^{-\pi i/18}$,
\begin{displaymath}
c_1=\zeta^3-\zeta^2-\zeta+1+(\zeta^2-\zeta+1)\lambda+(-\zeta^3+\zeta^2-1)\lambda^2,\quad\text{and}\quad c_2=\zeta^3-(\zeta-1)\lambda^2.
\end{displaymath}

The following lemma adds further detail to Lemma~\ref{lem:isotropystab}(c) and is easily checked.

\begin{lemm}\label{lem:Korbitc}The orbit under the finite group~$K$ of $M_c$ consists of
the eight mirrors $M_\alpha$ for $\alpha=c_{\scriptscriptstyle{\pm\pm\pm}}=\pm(r\pm1)(i\pm1)/2$ (so that
for example $c=c_{\scriptscriptstyle{+--}}$), and $\cM_A(O)$ is the set of these~$M_\alpha$'s. The 8~elements $k_\alpha\in K$
in the table below are such that $k_\alpha(M_c)=M_\alpha$.
\begin{center}
\begin{tabular}[t]{|c|c|c|c|c|c|c|c|c|}\hline
\vbox to 2.5ex{}$\alpha$&$c_{\scriptscriptstyle{+--}}$ & $c_{\scriptscriptstyle{--+}}$ & $c_{\scriptscriptstyle{---}}$&$c_{\scriptscriptstyle{+-+}}$&$c_{\scriptscriptstyle{-++}}$ & $c_{\scriptscriptstyle{-+-}}$ & $c_{\scriptscriptstyle{+++}}$ &$c_{\scriptscriptstyle{++-}}$\\[0.25ex]\hline
\vbox to 2.5ex{}$k_\alpha$&$1$ & $v$&$v^2$&$v^3$&$u^{-1}v^2u$&$vu^{-1}v^2u$& $v^2u^{-1}v^2u$&$v^3u^{-1}v^2u$\\[0.25ex]\hline
\end{tabular}
\end{center}
The orbit under~$K$ of $M_0$ consists of the 6 mirrors $M_\alpha$, $\alpha\in\{0,1,-1,i,-i,\infty\}$, and $\cM_B(O)$ is 
the set of these~$M_\alpha$'s. The 6~elements
$k_\alpha\in K$ in the table below satisfy $k_\alpha(M_0)=M_\alpha$.
\begin{center}
\begin{tabular}[t]{|c|c|c|c|c|c|c|}\hline
\vbox to 2.5ex{}$\alpha$&$0$&$i$&$-1$&$-i$&$1$&$\infty$ \\[0.25ex]\hline
\vbox to 2.5ex{}$k_\alpha$&$1$&$u j$&$vu j$&$v^2u j$&$v^3u j$&$u^{-1}v^2u j^6$\\[0.25ex]\hline
\end{tabular}
\end{center}
\end{lemm}

\subsection{}\label{sec:Pi} Cartwright and Steger discovered a very interesting torsion-free subgroup~$\Pi$ of $\bar\Gamma$ with finite index. 
The surface $\Pi\backslash\twoball$ is called the Cartwright-Steger surface in this article.

\begin{theo}[\cite{CS2}]\label{thm:pigp}
The elements 
\begin{displaymath}
a_1=vuv^{-1}j^4buv j^2,\quad
a_2=v^2u buv^{-1}uv^2j\quad\text{and}\quad
a_3=u^{-1}v^2u j^9bv^{-1}uv^{-1}j^8
\end{displaymath}
of~$\bar\Gamma$ generate a torsion-free subgroup~$\Pi$ of index~864, with $\Pi/[\Pi,\Pi]\cong\Z^2$.
\end{theo}

\begin{proof}Using the given presentation of~$\bar\Gamma$, the Magma {\tt Index} command
shows that $\Pi$ has index~864 in~$\bar\Gamma$. We see that $\Pi$ is torsion-free as follows.
The 864 elements $b^\mu k$, for $\mu=0,1,-1$ and $k\in K$, form a set of
representatives for the cosets $\Pi g$ of~$\Pi$ in~$\bar\Gamma$. One
can verify this by a method we shall use repeatedly:
for $g=b^\mu k$ and $g'=b^{\mu'}k'$, we check that $\Pi g\ne \Pi g'$
unless $\mu'=\mu$ and $k'=k$ by having Magma calculate the index in~$\bar\Gamma$
of $\langle a_1,a_2,a_3,g'g^{-1}\rangle$.

If $1\ne\pi\in\Pi$ has finite order, then $\pi=gtg^{-1}$ for one of the
elements~$t$ given in the table of Proposition~\ref{prop:torsionelts},
or the inverse of one of these. But then $(b^\mu k)t(b^\mu k)^{-1}\in\Pi$ for
some $\mu\in\{0,1,-1\}$ and $k\in K$, and Magma's {\tt Index\/} command shows
that this is not the case. 

The Magma {\tt AbelianQuotientInvariants} command shows that $\Pi/[\Pi,\Pi]\cong\Z^2$.
For any isomorphism $f:\Pi/[\Pi,\Pi]\to\Z^2$, the image under $f$ of~$a_1^3a_2^{-2}a_3^7$ is trivial. 
We can choose $f$ so that it maps $a_1$, $a_2$ and~$a_3$ to $(1,3)$, $(-2,1)$
and $(-1,-1)$, respectively. So $f(a_1a_2^{-1}a_3^2)=(1,0)$ and $f(a_1^{-1}a_2a_3^{-3})=(0,1)$.
\end{proof}

Magma shows that the normalizer of~$\Pi$ in~$\bar\Gamma$ contains~$\Pi$ as
a subgroup of index~3, and is generated by~$\Pi$ and~$j^4$. One may verify
that
\begin{displaymath}
\begin{aligned}
j^4a_1j^{-4}&=a_3a_2^{-3}a_3^3a_1,\\
j^4a_2j^{-4}&=a_3^{-1},\quad\text{and}\\
j^4a_3j^{-4}&=a_1^{-1}a_2^{-1}a_1a_2^2a_1^{-1}a_2^{-1}a_1a_3^{-1}a_1^{-1}a_2a_1.\\
\end{aligned}
\end{displaymath}
With the above isomorphism $f:\Pi/[\Pi,\Pi]\to\Z^2$, 
\begin{equation}\label{eq:j4actiononabelianization}
f(\pi)=(m,n)\quad\implies f(j^4\pi j^{-4})=(m,n)\begin{pmatrix}0&-1\\1&-1\end{pmatrix}\quad\text{for all}\ \pi\in\Pi.
\end{equation}


\subsection{}\label{sec:Picong} Cartwright and Steger noticed that the group $\Pi$ can be exhibited as a congruence subgroup of $\bar\Gamma$: we have two reductions $r_2:\Z[\zeta]\rightarrow \F_4=\F_2[\omega]$ and $r_3:\Z[\zeta]\rightarrow \F_9=\F_3[i]$ defined by sending $\zeta$ to $\omega$ (resp. $i$) where $1+\omega+\omega^2=0$ (resp. $i^2=-1$). They induce (surjective) group morphisms $\rho_2:\bar\Gamma\rightarrow \PU(3,\F_4)$ and $\rho_3:\bar\Gamma\rightarrow \PU(3,\F_9)$ (recall that $\PU(3,\F_4)$ and $\PU(3,\F_9)$ have respective cardinality 216 and 6048).

Note that for an element of $\PU(3,\F_4)$, the determinant is well defined since $\omega^3=1$. This enables us to define a (surjective) morphism ${\rm det}_2={\rm det}\circ \rho_2:\bar\Gamma\rightarrow \F_4^*$. Let us denote the subgroup ${\rm det}_2^{-1}(1)$ of index 3 of $\bar\Gamma$ by $\Pi_2$.

Remark also that there exist subgroups of order 21 in $\PU(3,\F_9)$ (they are all conjugate) and let us denote one of them by $G_{21}$. Then, define $\Pi_3:=\rho_3^{-1}(G_{21})$: it is a subgroup of $\oGamma$ of index $288=6048/21$.

So $\Pi_2\cap\Pi_3$ is a subgroup of $\oGamma$ of index $864=3\cdot 288$, and one can show that it is isomorphic to~$\Pi$.

\begin{lemm}\label{lem:numericalinv}
The Cartwright-Steger surface $X=\Pi\backslash\twoball$ has the following numerical invariants:
$$
c_1^2=9,\ \ c_2=3,\ \ \chi(\cO_X)=1,\ \ q:=h^{1,0}=1,\ \ p_g:=h^{2,0}=1,\ \ h^{1,1}=3.
$$
\end{lemm}
\begin{proof}
The orbifold $\bar\Gamma\backslash\twoball$ has orbifold Euler characteristic $1/288$ (see~\cite{PY} or \cite{Sa} for instance) so that $X$ has Euler characteristic $c_2(X)=3=864/288$. Then, as it is a two-ball quotient, $c_1^2(X)=9$ and thus its arithmetic genus is $\chi(\cO_X)=\frac1{12}(c_1^2+c_2)=1$.
Since $\Pi/[\Pi,\Pi]\cong\Z^2$, we have $b_1=2q=2$. So, from
\begin{eqnarray*}
1&=& \chi(\cO_X) =  1-q+p_g,\\
3&=& c_2(X) = 2b_0-2b_1+b_2,
\end{eqnarray*}
we deduce that $p_g=1$, $b_2=5$, and finally, $h^{1,1}=3$.
\end{proof}
We will see later (Corollary~\ref{coro:picard}) that the Picard number of $X$ is actually 3.
It is our purpose to understand the geometric properties of the surface $X$, especially using its Albanese map.

\section{Configurations of some totally geodesic divisors}\label{sec:summary}

Here we describe results about configuration of totally geodesic divisors on the Cart\-wright-Steger 
surface $X=\Pi\backslash B_{\bC}^2$.

Let $\projmap:X\rightarrow R={\oGamma}\backslash\twoball$ be the
projection. We use the notation of \S\ref{sec:descorb}. From the
description of the local groups at $P_1$, $P_2$ and $P_3$, we know
that $\projmap^{-1}(P_2)$ consists of $3=864/288$ points $O_1=\Pi (O),
O_2=\Pi (b\cdot O), O_3=\Pi (b^{-1}\cdot O)$ on $X$, $\projmap^{-1}(P_1)$
consists of $36=864/24$ points, and $\projmap^{-1}(P_3)$ consists of
$72=864/12$ points. It is easy using Magma to find $k_i\in K$ such
that $\projmap^{-1}(P_1)=\{\Pi(k_i.P):1\leqslant i\leqslant 36\}$.

For the curves $D_A$ and $D_B$, their preimages $ \projmap^{-1}(D_A)$ and
$\projmap^{-1}(D_B)$ consist of singular totally geodesic curves on $X$,
denoted to be of types~$A$ and~$B$ respectively.  By the description
of $R$ in Figure \ref{fig:barQR}, the curves can only have crossings
at $\projmap^{-1}(P_i)$ for $i=1,2,3$, and since these curves are totally
geodesic, these crossings are simple. It will be crucial for us to
know the genus of the irreducible components of these totally geodesic
curves, as well as the way they self-intersect and meet each other. In
this section, we explain how we can achieve this, using computer
calculations.

\subsection{} Our first step is to describe the groups $\bar\Gamma_0$ and~$\bar\Gamma_c$ of elements fixing~$M_0$ and~$M_c$, respectively.

As we saw in Lemma~\ref{lem:isotropystab}(d), $\bar\Gamma_0$ is a
central extension of a $(2,3,12)$-triangle group, with center of order~4.  
One may check that a presentation of $\bar\Gamma_0$ is given by
the generators $s_2=(jb)^{-1}$, $s_3=b$, $s_{12}=j$ and $z_0=v$ and
the relations
\begin{displaymath}
s_{12}^{12}=s_3^3=1,s_2^2=z_0^3,z_0^4=[s_{12},z_0]=[s_3,z_0]=[s_2,z_0]=s_{12}s_3s_2=1.
\end{displaymath}

We saw in Lemma~\ref{lem:isotropystab}(e) that $\bar\Gamma_c$ is a
central extension of a $(2,4,12)$-triangle group, with center of
order~3.  One may similarly check that a presentation of
$\bar\Gamma_c$ is given by the generators $t_2=(bu^{-1})^2$,
$t_4=j^{-1}(bu^{-1})^2$, $t_{12}=j$ and $z_c=u$ and the relations
\begin{displaymath}
 t_{12}^{12}=1,t_4^4=z_c,t_2^2=z_c^3=[t_{12},z_c]=[t_4,z_c]=[t_2,z_c]=t_{12}t_4t_2=1.
\end{displaymath}

\subsection{}
Let $\varphi:\twoball\to \XX$ be the natural map. If $M$ is a mirror
of type~$A$ or~$B$, let $\bar\Gamma_M$ denote the stabilizer of~$M$
(so $\bar\Gamma_\alpha=\bar\Gamma_{M_\alpha}$).  The group
$\Pi_M=\{\pi\in\Pi:\pi(M)=M\}=\Pi\cap\bar\Gamma_M$ acts on~$M$, and is
the fundamental group of the smooth curve $\Pi_M\backslash M$. The
embedding $M\hookrightarrow\twoball$ induces an immersion
$\varphi_M:\Pi_M\backslash M\to \XX$. We write $\Pi_\alpha$ instead of~$\Pi_{M_\alpha}$. 
We need now to describe $\Pi_M$, and we start by the simpler case of mirrors of type $B$.

\subsection{The groups $\Pi_M$ when $M$ is a mirror of type~$B$.}\label{sec:MirrortypeB}
First, we consider
$\Pi_0=\Pi_{M_0}=\{\pi\in\Pi:\pi(M_0)=M_0\}=\Pi\cap\bar\Gamma_0=\Pi_2\cap\Pi_3\cap\bar\Gamma_0={\rm
  det}_2^{-1}(1)\cap\rho_3^{-1}(G_{21})\cap\bar\Gamma_0$ by
\S\ref{sec:Picong}. Restricting ${\rm det}_2$ and $\rho_3$ to
$\bar\Gamma_0$, Magma finds that $\Pi_0$ has index 288 in
$\bar\Gamma_0$.
\begin{prop}\label{prop:pi0prop}The group $\Pi_0$ has a presentation
\begin{equation}\label{eq:rank4presentation}
\langle u_1,\ldots,u_4,v_1,\ldots,v_4\ :\ [u_1,v_1][u_2,v_2][u_3,v_3][u_4,v_4]=1\rangle,
\end{equation}
with generators $u_i$, $v_i$, given below, and so $\Pi_0\backslash M_0$ 
is a curve of genus~4.
\end{prop}
\begin{proof}As $j^4$ normalizes~$\Pi$, we can define $g_1,\ldots,g_8\in\Pi$ by setting
$g_1=a_3^{-3}a_1^{-1}a_2a_1$, $g_3=a_2a_1^{-2}a_3^{-3}a_1^{-1}$, $g_5=j^4a_2a_1j^8a_2^{-1}a_3^3a_1^2$,
and $g_7=j^4a_1^{-1}a_2^{-1}j^4a_2a_1j^4$, and then $g_{2\nu}=j^4g_{2\nu-1}j^{-4}$ for $\nu=1,2,3,4$.
These are in~$\bar\Gamma_0$.
Magma verifies that $G=\langle g_1,\ldots,g_8\rangle$ has index~288 in~$\bar\Gamma_0$,
and so $G=\Pi_0$, and gives a presentation of~$\Pi_0$ which has just one relation:
\begin{equation}\label{eq:pi0relation}
g_1g_2g_3g_4g_5g_6g_7g_8g_1^{-1}g_3^{-1}g_5^{-1}g_7^{-1}g_2^{-1}g_4^{-1}g_6^{-1}g_8^{-1}=1.
\end{equation}
By a method shown
to us by Jonathan Hillman, we replace the generators $g_i$ by generators
$u_i$ and~$v_i$, where $u_i=E_1\cdots E_{i-1}D_iE_{i-1}^{-1}\cdots E_1^{-1}$ and
$v_i=E_1\cdots E_{i-1}E_iE_{i-1}^{-1}\cdots E_1^{-1}$,
where
\begin{displaymath}
\begin{aligned}
D_1&=g_1g_2g_3g_4g_5g_6g_7,\\
D_2&=g_1g_2g_3g_4,\\
\end{aligned}
\quad
\begin{aligned}
D_3&=g_1,\\      
D_4&=g_3^{-1},\\
\end{aligned}
\quad\text{and}\quad
\begin{aligned}
E_1&=g_8g_1^{-1}g_3^{-1}g_5^{-1},\\ 
E_2&=g_5g_6g_2^{-1},\\ 
\end{aligned}
\quad
\begin{aligned}
E_3&=g_2g_3g_6^{-1},\\ 
E_4&=g_6\\
\end{aligned}
\end{displaymath}
and these generators $u_i$ and $v_i$ satisfy the stated relation.
\end{proof}

We now consider $\Pi_M$ for the other mirrors~$M$ of type~$B$.
\begin{prop}\label{prop:typeBmirrorgroups}If $g\in\bar\Gamma$ and $M=g(M_0)$ is a mirror of type~$B$, then
\begin{itemize}
\item[(a)]There is a $\pi\in\Pi$ such that $\pi(M)=M_0$, $M_1$ or~$M_\infty$.
\item[(b)]Correspondingly, $\Pi_M$ is conjugate in~$\Pi$ to either $\Pi_0$, $\Pi_1$ or~$\Pi_\infty$.
\item[(c)] $\Pi_M=g\Pi_0g^{-1}$.
\item[(d)] $h(\Pi_M)h^{-1}=\Pi_{h(M)}$ for any $h\in\bar\Gamma$.
\end{itemize}
In particular, it follows from (c) that for any mirror $M$ of type $B$, $\Pi_M\backslash M\cong\Pi_0\backslash M_0$.
\end{prop}
\begin{proof}
(a) The elements $b^\mu k$, $\mu=0,1,-1$ and $k\in K$, form a set of coset representatives 
of~$\Pi$ in~$\bar\Gamma$. So using Lemma~\ref{lem:Korbitc}, we may assume that $M=b^\mu(M_\alpha)$ for some $\mu\in\{0,1,-1\}$ 
and $\alpha\in\{0,\pm1,\pm i,\infty\}$. Then, searching amongst short words in the generators $a_i$ of~$\Pi$, we 
quickly find $\pi\in\Pi$ such that $\pi(M)=M_\beta$ for $\beta\in\{0,1,\infty\}$. For example,
taking $\pi=a_3^3a_1^2a_2^{-1}$, we have $\pi(bM_{-1})=M_1$.
This proves~(a), and (b) follows immediately.

(c) We first show that $h\Pi_0h^{-1}\subset\Pi$ for each $h\in\bar\Gamma$. We may assume that $h=b^\mu k$ as in~(a). 
For each of the 8~generators $g_j$ of~$\Pi_0$ given in the proof of Proposition~\ref{prop:pi0prop} we have Magma 
check that $\langle a_1,a_2,a_3,hg_jh^{-1}\rangle$ has index~864 in~$\bar\Gamma$ , so that $hg_jh^{-1}\in\Pi$.
It follows, in particular, that $h\Pi_0h^{-1}=\Pi_0$ for each $h\in\bar\Gamma_0$.
We next prove~(c) in the cases $g=k_\beta$, $\beta=1,\infty$. Now $k_\beta\Pi_0k_\beta^{-1}\subset\Pi$ and so 
$\Pi_0\subset k_\beta^{-1}\Pi_\beta k_\beta\subset\bar\Gamma_0$.
Choose a transversal $t_1=1,\ldots,t_{288}$ of~$\Pi_0$ in~$\bar\Gamma_0$. Then Magma verifies
that the index in~$\bar\Gamma$ of $\langle a_1,a_2,a_3,k_\beta t_ik_\beta^{-1}\rangle$ is less than~864 if $i\ne1$.
Thus $\Pi_0=k_\beta^{-1}\Pi_\beta k_\beta$, and (c) holds for $g=k_\beta$, $\beta=0,1,\infty$.
By~(a), for our given~$g$, there is a $\pi\in\Pi$ so that $g(M_0)=\pi(M_\beta)$ for one of these~$\beta$'s.
Then $h=k_\beta^{-1}\pi^{-1}g$ is in~$\bar\Gamma_0$, so that $h\Pi_0h^{-1}=\Pi_0$. Then $(\pi^{-1}g)\Pi_0(\pi^{-1}g)^{-1}=\Pi_\beta$
by~(c) for $g=k_\beta$. Thus $g\Pi_0g^{-1}=\pi(\Pi_{M_\beta})\pi^{-1}=\Pi_{\pi(M_\beta)}=\Pi_M$.
Part~(d) is immediate from~(c).
\end{proof}

The three possibilities in~(a) are mutually exclusive (see \S\ref{sec:triangledesc} below). If $M$ is a mirror of type~$B$,
then by Proposition~\ref{prop:typeBmirrorgroups}(a), the image of the immersion 
$\varphi_M:\Pi_M\backslash M\to \XX$ is equal to
the image of $\varphi_{M'}$ for $M'=M_0$, $M_\infty$, or~$M_1$. We will denote by $E_1$, $E_2$ and~$E_3$ respectively these images (which are distinct since the cases are mutually exclusive).
To calculate entries in the table in \S\ref{sec:imhone}, we need explicit generators for $\Pi_\infty$.
We start with the generators $g_i''=k_\infty g_ik_\infty^{-1}$, where $g_1,\ldots,g_8$ are as in proof of 
Proposition~\ref{prop:pi0prop}. The $g_i''$ satisfy exactly the same relation
as do the $g_i$'s, and so standard generators $u_i$ and~$v_i$ can be found for $\Pi_\infty$ in exactly
the same way as was done for~$\Pi_0$. To calculate the  $f(u_i)$ and $f(v_i)$'s, we need to 
express the $g_i''$'s in terms of the generators of~$\Pi$. One may verify that:
\begin{displaymath}
\begin{aligned}
g_1''&=j^4(a_1^{-1}a_3^{-2}a_1^{-1})j^8a_1^{-1}a_2^{-1},\\
g_3''&=j^8(a_3a_1a_2a_1^{-1}a_2^{-1})j^4,\\
\end{aligned}
\quad
\begin{aligned}
g_5''&=j^8(a_2^{-1}a_3^{-1})j^4,\\
g_7''&=j^4(a_1a_3a_1^{-1}a_3^{-2})j^8,\\
\end{aligned}
\end{displaymath}
and $g_{2\nu}''=j^4g_{2\nu-1}''j^8$ for $\nu=1,2,3,4$.

\subsection{The groups $\Pi_M$ when $M$ is a mirror of type~$A$}\label{sec:MirrortypeA}
Magma finds that $\Pi_c$ has index~324 in $\bar\Gamma_c$.

\begin{prop}\label{prop:picprop}The group $\Pi_c$ has a presentation
\begin{equation}\label{eq:rank10presentation}
\langle u_1,\ldots,u_{10},v_1,\ldots,v_{10}\ :\ [u_1,v_1][u_2,v_2]\cdots[u_9,v_9][u_{10},v_{10}]=1\rangle,
\end{equation}
and so $\Pi_c\backslash M_c$ is a curve of genus~10.
\end{prop}
\begin{proof}The proof is very similar to that of Proposition~\ref{prop:pi0prop}.
We define 20~elements $g_1,\ldots,g_{20}$ of~$\Pi$ by setting
\begin{displaymath}
\begin{aligned}
g_1&=j^8a_1^{-1}a_2a_1a_3a_1^{-1}j^4a_2a_1,\\
g_3&=j^4a_2a_1a_2^{-2}a_1^{-1}a_3j^4a_3^3j^4,\\
g_5&=j^8a_1^{-1}j^4a_2a_1j^4a_3a_2^{-1}a_1a_3a_1^{-1}j^8,\\
g_7&=j^8a_2a_1j^4a_3^{-1}j^4a_2a_1^{-1}a_2^{-1}a_3^{-3}j^8,\\
g_9&=j^8a_1^{-1}a_2^{-2}a_1^{-1}a_3^{-1}j^8a_1^{-1}a_2^{-1}j^8,\\
\end{aligned}
\quad
\begin{aligned}
g_{12}&=a_2^{-1}a_1a_3a_1^{-1}a_3^{-1}j^4a_3a_1a_2^2a_1^{-1}a_2^{-1}j^8,\\
g_{15}&=j^4a_1j^4a_2a_3a_1^{-1}j^4,\\
g_{17}&=j^8a_1^{-2}a_2^{-1}j^4a_3a_1a_2a_1,\\
g_{19}&=a_2^{-1}a_1a_3a_1^{-1}a_3^{-2}j^4a_1a_2j^4a_1^{-1}a_2^{-1}j^4,\\
\vphantom{g_{20}}&\vphantom{a_1^{-1}}\\
\end{aligned}
\end{displaymath}
and also $g_{\nu+1}=j^4g_{\nu}j^{-4}$ for $\nu\in\{1,3,5,7,9,10,12,13,15,17,19\}$. These are in~$\bar\Gamma_c$.
Magma verifies that $G=\langle g_1,\ldots,g_{20}\rangle$ has index~324 in~$\bar\Gamma_c$,
and so $G=\Pi_c$, and gives a presentation of~$\Pi_c$ which has just one relation:
\begin{displaymath}
\begin{aligned}
&g_4g_{14}^{-1}g_2^{-1}g_{17}^{-1}g_9g_{19}g_{20}g_{14}g_7^{-1}g_{10}^{-1}g_5^{-1}g_{16}^{-1}g_3^{-1}g_{12}^{-1}g_1g_2g_{18}^{-1}g_{10}g_{19}^{-1}g_{12}\\
&\times g_8^{-1}g_{11}^{-1}g_6^{-1}g_{15}g_{16}g_4^{-1}g_{13}^{-1}g_1^{-1}g_{17}g_{18}g_{11}g_{20}^{-1}g_{13}g_7g_8g_9^{-1}g_5g_6g_{15}^{-1}g_3=1.
\end{aligned}
\end{displaymath}
Using the same method as in the proof of Proposition~\ref{prop:pi0prop}, we can replace the generators $g_i$ by generators
$u_i$ and~$v_i$ satisfying the given relation. We omit the details.
\end{proof}

We now consider $\Pi_M$ for the other mirrors~$M$ of type~$A$. As well as $c=c_{\scriptscriptstyle{+--}}$,
the parameter $-c=c_{\scriptscriptstyle{---}}$ is important in the next result.

\begin{prop}\label{prop:typeAmirrorgroups}If $g\in\bar\Gamma$ and $M=g(M_c)$ is a mirror of type~$A$, then
\begin{itemize}
\item[(a)]There is a $\pi\in\Pi$ such that $\pi(M)=M'$, where $M'\in\{M_c, M_{-c}, b(M_c),b^{-1}(M_c)\}$.
\item[(b)]If $M'$ is as in~(a), then $\Pi_M$ is conjugate in~$\Pi$ to $\Pi_{M'}$.
\item[(c)]$\Pi_M=g\Pi_cg^{-1}$ in the first two cases of~(a), and in particular if $g=k_\alpha$ for any 
$\alpha\in\{c_{\scriptscriptstyle{+++}},\ldots,c_{\scriptscriptstyle{---}}\}$, so that $\Pi_\alpha=k_\alpha\Pi_ck_\alpha^{-1}$ for all these $\alpha$'s. 
\item[(d)]In the other two cases of~(a), $g\Pi_cg^{-1}$ has index~3 in~$\Pi_M$.
\end{itemize}
\end{prop}
\begin{proof}The proof is similar to that of Proposition~\ref{prop:typeBmirrorgroups}. For~(a), we may assume that 
$M=b^\mu(M_\alpha)$ for some $\mu\in\{0,1,-1\}$ 
and $\alpha\in\{c_{\scriptscriptstyle{+++}},\ldots,c_{\scriptscriptstyle{---}}\}$. For each of these~$M$'s, we find explicit $\pi\in\Pi$ such that $\pi(M)=M'$ 
for an~$M'$ in the given list. The most complicated~$\pi$ needed is $\pi=a_2a_1^{-2}a_3^{-1}a_1a_3^{-1}a_1^{-1}a_2^{-2}$, satisfying
$\pi(b^{-1}(M_{c_{\scriptscriptstyle{-+-}}}))=b(M_c)$.

In proving~(c) and~(d), we first show that $h\Pi_ch^{-1}\subset\Pi$ for all $h\in\bar\Gamma$ as in
Proposition~\ref{prop:typeBmirrorgroups}, and therefore that $h\Pi_ch^{-1}=\Pi_c$ for $h\in\bar\Gamma_c$.
We are reduced to proving~(c) and~(d) for $g=k_c$, $k_{-c}$, $b$ and~$b^{-1}$. We have $\Pi_c\subset g^{-1}\Pi_Mg\subset\bar\Gamma_c$, and choose a transversal $t_1=1,\ldots,t_{324}$ of~$\Pi_c$ in~$\bar\Gamma_c$.
Magma verifies that $gt_ig^{-1}\in\Pi$ only for $i=1$ in the first two cases as in Proposition~\ref{prop:typeBmirrorgroups}, but that $gt_ig^{-1}\in\Pi$ for three $i$'s in the last two cases.
\end{proof}

In Proposition~\ref{prop:typeAmirrorgroups}(d), $\Pi_M\backslash M$ has genus~4
by the Riemann-Hurwitz formula, and we can find explicit generators $u_i$, $v_i$ of~$\Pi_M$
such that $[u_1,v_1][u_2,v_2][u_3,v_3][u_4,v_4]=1$. When $M=b(M_c)$, the following eight elements generate $\Pi_M$:
\begin{displaymath}
\begin{aligned}
p_1&=a_2^3a_1^{-1}a_3^{-1}j^8a_2^{-2}a_1^{-1}j^4,\\
p_2&=a_3^3a_1a_3^2a_2a_1j^4a_3^{-1}j^8a_3^{-2}a_1^{-1}a_3^{-3},\\
p_3&=j^8a_1^{-1}a_3^{-3}a_2^2j^4a_3^{-2}a_1^{-1}a_3^{-3},\\
p_4&=j^8a_2a_1a_2^{-2}a_1^{-1}j^4a_3^3a_1^2a_2^{-1},\\
\end{aligned}
\quad
\begin{aligned}
p_5&=a_3^3a_1a_3^2j^4a_1^{-1}j^8a_3^2a_1a_2^{-3},\\
p_6&=a_3^3a_1a_2a_1a_3a_2^{-3},\\
p_7&=a_3^3a_1j^8a_1a_2^{-2}a_1^{-1}a_3^2j^4,\\
p_8&=j^4a_3^{-2}j^8a_2a_1a_2a_1a_2^{-2},\\
\end{aligned}
\end{displaymath}
and satisfy the single relation
\begin{displaymath}
p_5^{-1}p_2^{-1}p_5p_1p_3p_8^{-1}p_4p_1^{-1}p_7^{-1}p_6^{-1}p_7p_2p_3^{-1}p_8p_4^{-1}p_6=1.
\end{displaymath}
Following the same procedure as in the proof of Proposition~\ref{prop:pi0prop}, we obtain
a presentation~\eqref{eq:rank4presentation} for $\Pi_M$, with
$u_i=E_1\cdots E_{i-1}D_iE_{i-1}^{-1}\cdots E_1^{-1}$ and $v_i=E_1\cdots E_{i-1}E_iE_{i-1}^{-1}\cdots E_1^{-1}$
for
\begin{displaymath}
\begin{aligned}
D_1&=p_5^{-1}p_2^{-1}p_5p_1p_3p_8^{-1}p_4p_1^{-1}p_7^{-1},\\
D_2&=p_5^{-1}p_2^{-1}p_5p_1p_3p_8^{-1},\\
\end{aligned}
\
\begin{aligned}
D_3&=p_5^{-1}p_2^{-1}p_5p_1,\\
D_4&=p_5^{-1},\\
\end{aligned}
\quad\text{and}\quad
\begin{aligned}
E_1&=p_6^{-1},\\ 
E_2&=p_4p_1^{-1}p_2p_3^{-1},\\ 
\end{aligned}
\ 
\begin{aligned}
E_3&=p_3,\\          
E_4&=p_2^{-1}.\\
\end{aligned}
\end{displaymath}

The four possibilities in~Proposition~\ref{prop:typeAmirrorgroups}(a) are mutually exclusive
(see \S\ref{sec:triangledesc}). If $M$ is a mirror of type~$A$,
then by Proposition~\ref{prop:typeAmirrorgroups}(a), the image of the immersion 
$\varphi_M:\Pi_M\backslash M\to \XX$ is equal to
the image of $\varphi_{M'}$ for $M'=b(M_c)$, $b^{-1}(M_c)$, $M_c$, or~$M_{-c}$. We will denote by $C_1,C_2,C_3$ and $C_4$ respectively these images. Again, they are distinct.

\subsection{}\label{sec:triangledesc}In \S\ref{sec:MirrortypeB} and~\S\ref{sec:MirrortypeA}, we have identified 7 distinct irreducible totally geodesic curves in~$\XX$, 
4 of type $A$, and 3 of type~$B$. Just the knowledge of the indices of the groups $\Pi_M$ in $\bar\Gamma_M$ together with Lemma~\ref{lem:isotropystab}(d) and (e) 
enables us to determine the genus of the curves $\Pi_M\backslash M$.

For instance, since $\Pi_M$ has index $288$ in $\bar\Gamma_M$ when $M$ is of type $B$, and since the center of $\bar\Gamma_M$ has order $4$, the normalization $\hat E_i$ of the curve $E_i$ is an orbifold covering of degree $72=288/4$ of the orbifold $D_B\cong\P^1_\C$ endowed with three orbifold points $(P_4,P_3,P_2)$ of respective multiplicities $(2,3,12)$ hence by the Riemann-Hurwitz formula, its genus is indeed
$$
g(\hat E_i)= \frac {72}2 \left(-2 + \frac{2-1}2 + \frac{3-1}3+\frac{12-1}{12} \right)+1=4.
$$
Note that $864=4\cdot 3\cdot 72$, where 4 is the order of the reflections of type $B$ and 3 the number of curves of type $B$, so that the three possibilities in~Proposition~\ref{prop:typeBmirrorgroups}(a) are mutually exclusive.

In the same way, the normalizations of $C_1$ and $C_2$ (resp. $C_3$
and $C_4$) are orbifold coverings of degree $36$ (resp. $108$) of the
orbifold~$D_A$ whose normalization is $\P^1_\C$, endowed with three
orbifold points $(P_1,P_3,P_2)$ of respective multiplicities
$(2,4,12)$ so that $g(\hat C_1)=g(\hat C_2)=4$ and $g(\hat C_3)=g(\hat C_4)=10$. 
Here again, $864 = 3(2\cdot 36+2\cdot 108)$ where $3$ is the order of the reflections of type~$A$, hence the four possibilities
in~Proposition~\ref{prop:typeAmirrorgroups}(a) are mutually exclusive.

However, we will need to know explicit generators of the various groups $\Pi_M$ (see below).
 
\subsection{}\label{sec:self} Now, we want to find out how the curves $C_i$ and $E_i$ self intersect. The next result is a straightforward consequence of the discussion at the beginning of \S\ref{sec:summary}.
\begin{lemm}\label{lem:selfintersect}Suppose that $x\in\XX$ is the image under~$\varphi_M$
of two or more distinct elements of~$\Pi_M\backslash M$. If $M$ is of type~$B$, then $x$ must be 
one of the three points $\Pi(O)$, $\Pi(b.O)$ and~$\Pi(b^{-1}.O)$. If $M$ is of type~$A$, then
$x$ is either one of these three points or one of the 36~points $\Pi(k_i.P)$, where the
$k_i$ are as above.
If $\xi\in M$, then $\varphi_M(\Pi_M\xi)$ is one of the three points $\Pi(b^\mu.O)$, $\mu=0,1,-1$, if and only if $\xi$ is in
the $\bar\Gamma$-orbit of~$O$, and it is one of the 36~points $\Pi(k_i.P)$ if and only if $\xi$ is in
the $\bar\Gamma$-orbit of~$P$. 
\end{lemm}

\begin{lemm}\label{lem:preimagecount}
\begin{itemize}
\item[(i)] For each mirror $M$ of type~$B$, there are exactly six distinct $\Pi_M\xi\in\Pi_M\backslash M$
such that $\xi\in M$ is in the $\bar\Gamma$-orbit of~$O$.
\item[(ii)] Suppose that $M$ is a mirror of type~$A$, and that there
is a $\pi\in\Pi$ such that $\pi(M)=M_c$ or~$M_{-c}$,
respectively such that $\pi(M)=b(M_c)$ or~$b^{-1}(M_c)$. There are exactly~9 (respectively~3) distinct $\Pi_M\xi\in\Pi_M\backslash M$
such that $\xi\in M$ is in the $\bar\Gamma$-orbit of~$O$. There are exactly~54 (respectively~18) distinct 
$\Pi_M\xi\in\Pi_M\backslash M$ such that $\xi\in M$ is in the $\bar\Gamma$-orbit of~$P$.
\end{itemize}
\end{lemm}

\begin{proof}
(i) This follows from the description given in \S\ref{sec:triangledesc}. Indeed, the orbifold point $P_2$ on $D_B$ has weight $12$ so that it has $6=72/12$ preimages in $\hat E_i$.

(ii) In the same way, the orbifold point $P_2$ on $D_A$ has weight $12$ so that it has $9=108/12$ preimages in $\hat C_3$ and $\hat C_4$ (resp. $3=36/12$  preimages in $\hat C_1$ and $\hat C_2$). Also, the orbifold point $P_1$ has weight $2$ so that it has $54=108/2$ preimages in $\hat C_3$ and $\hat C_4$ (resp. $18=36/2$  preimages in $\hat C_1$ and $\hat C_2$).
\end{proof}

For any mirror~$M$, and any $\mu\in\{0,1,-1\}$, let
\begin{displaymath}
n_\mu(M)=\sharp\{\Pi_M\xi\in\Pi_M\backslash M:\varphi_M(\Pi_M\xi)=\Pi(b^\mu.O)\}.
\end{displaymath}
By Lemma~\ref{lem:preimagecount}(i), $n_0(M)+n_1(M)+n_{-1}(M)=6$ if $M$ is of type~$B$.

\begin{prop}\label{prop:zeroimmersions}If $M$ is a mirror of type~$B$, then 
according to the three possibilities in Proposition~\ref{prop:typeBmirrorgroups}(a),
$(n_0(M),n_1(M),n_{-1}(M))$ is either $(3,1,2)$, $(1,4,1)$ or~$(2,1,3)$, respectively.
\end{prop}

\begin{proof}This is easily seen by choosing representatives $\gamma\in\bar\Gamma_M$ of the 6 
distinct double cosets $\Pi_M\gamma(K\cap \bar\Gamma_M)$ for 
$M=M_0$, $M_1$ and $M_\infty$ and then computing their images $\Pi\gamma K$ in $\Pi\backslash\bar\Gamma/K=\{\Pi K,\Pi b K,\Pi b^{-1} K\}$.
\end{proof}

That $n_0(M_0)$, $n_0(M_1)$, $n_0(M_\infty)$ are distinct gives another proof that the
images of~$\varphi_{M_0}$, $\varphi_{M_1}$ and~$\varphi_{M_\infty}$ are distinct and that the 
cases in Proposition~\ref{prop:typeBmirrorgroups}(a) are mutually exclusive.

We now calculate $n_\nu(M)$, $\nu=0,1,-1$, for mirrors~$M$ of type~$A$, as well as the
numbers 
\begin{displaymath}
m_i(M)=\sharp\{\Pi_M\xi\in\Pi_M\backslash M:\varphi_M(\Pi_M\xi)=\Pi(k_i.P)\}
\end{displaymath}
for $i=1,\ldots,36$ (recall that the $k_i$'s were defined at the beginning of \S\ref{sec:summary}). If $\pi\in\Pi$ and $M'=\pi(M)$, then 
$n_\nu(M')=n_\nu(M)$ and $m_i(M')=m_i(M)$ for each~$\nu$ and~$i$, and
so by Proposition~\ref{prop:typeAmirrorgroups}(a), we need only do the calculation
for $M_c$, $M_{-c}$, $b(M_c)$ and~$b^{-1}(M_c)$.
\begin{prop}\label{prop:intersectioncount}For mirrors $M$ of type~$A$, $(n_0(M),n_1(M),n_{-1}(M))$ is
$(4,3,2)$ for the first two cases in Proposition~\ref{prop:typeAmirrorgroups}(a), and~$(0,1,2)$ for 
the other two.

For a suitable ordering of the $k_i$, the numbers $m_i=m_i(M)$ are as follows:
\begin{center}
\begin{tabular}[t]{|c|c|c|c|c|}\hline
\vbox to 2.0ex{}$M$&$m_1\ldots,m_{12}$&$m_{13},\ldots,m_{18}$&$m_{19},\ldots,m_{24}$&$m_{25},\ldots,m_{36}$\\[0.25ex]\hline
\vbox to 2.0ex{}$M_c$&2&0&3&1\\[0.25ex]\hline
\vbox to 2.0ex{}$M_{-c}$&2&3&0&1\\[0.25ex]\hline
\vbox to 2.0ex{}$b(M_c)$&0&0&1&1\\[0.25ex]\hline
\vbox to 2.0ex{}$b^{-1}(M_c)$&0&1&0&1\\[0.25ex]\hline
\end{tabular}
\end{center}
\end{prop}

\begin{proof}
As in Proposition~\ref{prop:zeroimmersions}, to get the
numbers $n_\nu(M)$, we choose representatives $\gamma\in\bar\Gamma_M$ of the 9 (resp.~3) distinct
double cosets $\Pi_M\gamma(\bar\Gamma_M\cap K)$
(respectively, $\Pi_M\gamma(\bar\Gamma_M\cap b^j Kb^{-j})$)
for $M=M_c$ and $M_{-c}$ (respectively, $M=b^j(M_c)$, $j=1,-1$) and then
compute their images $\Pi\gamma K$ (respectively, $\Pi\gamma b^jK$) in $\Pi\backslash\bar\Gamma/K$.

To compute the numbers $m_i(M)$ for $M=M_\alpha$, $\alpha=c,-c$, (respectively, $M=b^j(M_c)$, $j=1,-1$), we choose representatives
$\gamma\in\bar\Gamma_M$ of the 54 (respectively, 18) distinct double cosets $\Pi_M\gamma(\bar\Gamma_M\cap k_\alpha\bar\Gamma_Pk_\alpha^{-1})$,
(respectively, $\Pi_M\gamma(\bar\Gamma_M\cap b^j\bar\Gamma_Pb^{-j})$) and compute their images $\Pi\gamma k_\alpha\bar\Gamma_P$ 
(respectively, $\Pi\gamma b^j\bar\Gamma_P$) in~$\Pi\backslash\bar\Gamma/\bar\Gamma_P$. 
\end{proof}


%
%
%
\subsection{}\label{sec:config} The knowledge of the numbers $n_\nu(M)$ and $m_i(M)$ determines how our seven totally geodesic curves self-intersect. In order to determine how two distinct such curves intersect, we also need to know which of the $72$ points of $\projmap^{-1}(P_3)$ each of them contains. Using exactly the same method as in Propositions~\ref{prop:zeroimmersions}~and~\ref{prop:intersectioncount}, we obtain

\begin{prop}\label{prop:xi12orbit1}There are exactly 72 distinct points in $\projmap^{-1}(P_3)$. The set of these points
may be partitioned into three subsets of size~24, consisting of the points in the
images of~$M_0$, $M_1$ and~$M_\infty$, respectively. For $\alpha=0,1,\infty$, the set of
24 points belonging to the image of~$M_\alpha$ is
partitioned into sets of $n_1$, $n_2$, $n_3$ and~$n_4$ points in the images of~$M_c$, $M_{-c}$, $b(M_c)$
and~$b^{-1}(M_c)$, respectively, where $(n_1,n_2,n_3,n_4)=(6,6,6,6)$ for $\alpha=0$,
$(n_1,n_2,n_3,n_4)=(9,9,3,3)$ for $\alpha=1$, 
and $(n_1,n_2,n_3,n_4)=(12,12,0,0)$ for $\alpha=\infty$.
\end{prop}

\subsection{}\label{sec:imhone}  We have seen (cf. \S\ref{sec:Pi}) that $H_1(X,\bZ)=\bZ e_1+\bZ e_2\cong \bZ^2$ in terms of a basis $e_1$ and~$e_2$.
For each of the genus~4 curves $D=E_i$, $i=1,2,3$, and $D=C_j$, $j=1,2$, a presentation~\eqref{eq:rank4presentation} 
can be given for~$\pi_1(\hD)$. Abusing notation, we denote by $f:H_1(\hD,\bZ)\rightarrow H_1(X,\bZ)\cong\bZ^2$ the
homomorphism induced by the normalization of the immersed image of~$D$ in~$X$.
For $E_1$, $E_2$ and~$C_1$ (which is all we need for later computations) we have given
generators $u_i$, $v_i$, $i=1,\dots,4$, of~$\pi_1(\hD)$ explicitly as words in the generators $a_1$, $a_2$ 
and~$a_3$ of~$\Pi$. So it is routine to compute their images $f(u_i)$, $f(u_j)$ in~$H_1(X,\bZ)$ in terms of $e_1$, $e_2$.  
We obtain:
{\small
\begin{center}
\begin{tabular}[t]{|c|c|c|c|c|c|c|c|c|}\hline
\vbox to 2.25ex{}$D$&$f(u_1)$&$f(v_1)$&$f(u_2)$&$f(v_2)$&$f(u_3)$&$f(v_3)$&$f(u_4)$&$f(v_4)$\hfil\\[0.25ex]\hline
\vbox to 2.25ex{}$E_1$&$(-5,-2)$&$(-2,7)$&$(-2,1)$&$(0,0)$&$(1,4)$&$(3,-6)$&$(2,5)$&$(-1,-4)$\\[0.25ex]\hline
\vbox to 2.25ex{}$E_2$&$(-1,2)$&$(2,-1)$&$(-2,1)$&$(0,0)$&$(-3,0)$&$(-1,2)$&$(-2,1)$&$(3,0)$\\[0.25ex]\hline
\vbox to 2.25ex{}$C_1$&$(0,-2)$&$(-2,0)$&$(-4,0)$&$(0,2)$&$(-4,2)$&$(4,0)$&$(2,0)$&$(0,-2)$\\[0.25ex]\hline
\end{tabular}
\end{center}
}
Of course, we can also compute the image under~$f$ of the generators of the fundamental group of the genus~10 curves $\hat C_3$ and $\hat C_4$.


\section{Picard number}

\begin{lemm}\label{lem:totalgeod}
Suppose $D$ is a reduced (not necessarily irreducible) totally geodesic curve on a smooth complex two-ball quotient $X$
self-intersecting only at $P_1,\dots, P_k$ with simple multiplicities given by
$(b_1,\cdots, b_k)$ and let us denote by $D_i$ ($i=1,\dots,n$) its irreducible components, $\hD_i$ their normalization. Let $\nu:\hD=\cup_i\hD_i\rightarrow D$ be the normalization of $D$.  Then 
$$
K_X\cdot D=3\sum_{i=1}^n(g(\hD_i)-1)\ \ \mbox{and}\ \ D\cdot D=\frac12 e(\hD)+\tdelta, \ \ \ \mbox{where} \ \ \tdelta=\sum_{i=1}^k b_i(b_i-1)
$$
and $e(\hD)$ is the Euler characteristic of $\hD$.
\end{lemm}

\begin{proof}Note that we are in the case of a (non necessarily connected) immersed smooth curve in a surface, 
with singularities given by intersections of transversal local
branches. Moreover, it is well known that for a totally geodesic curve
$D$ in a two-ball quotient, $c_1(K_{\hD})=\frac23\nu^*c_1(K_X)$ (this
is a simple computation involving the curvature form on $\twoball$).
As a consequence, by the adjunction formula,
$$
K_X\cdot D=\int_D c_1(K_X)=\frac 32\sum_i \int_{\hD_i} c_1(K_{\hD_i})= 3\sum_{i=1}^n(g(\hD_i)-1).
$$
Recall moreover from~\cite[\S II.11]{BHPV} that
$$
g(D)=g(\hD)+\dan(D),\ \ \ \mbox{where}\ \  g(\hD)=1+\sum_i(g(\hD_i)-1)\ \ \mbox{and}\ \ \dan(D)=\sum_{x\in D}\dim_{\bC}(\nu_*\cO_{\hD}/\cO_D)
$$
(here, the genus of a singular curve is its arithmetic genus).
From the adjunction formula for embedded curves, $2(g(D)-1)=K_X\cdot D+D\cdot D$ and therefore,
$$
D\cdot D=2(g(D)-1)-K_X\cdot D=2(g(\hD)+\dan(D)-1)-3(g(\hD)-1)=(1-g(\hD))+2\dan(D).
$$
Finally, observe that in the case at hand, $\dan(D)=\frac12\sum_{i=1}^k b_i(b_i-1)=\frac12\tdelta$.
\end{proof}

\begin{lemm}\label{lem:intersectionnumbers}
We have the following intersection numbers.
\begin{itemize}
\item[(a)] For $i=1, 2, 3$, we have $K_X\cdot E_i=9$.  Moreover,  for $i=1,2$, $E_i\cdot E_i=5$, $E_i\cdot E_3=9$ and $E_1\cdot E_2=13$. We also have  $E_3\cdot E_3=9$. 
\item[(b)] Denote by $C$ either $C_1$ or $C_2$.  Then $K_X\cdot C=9$, $C\cdot C=-1$,   
$E_1\cdot C=11$, $E_2\cdot C=7$ and $E_3\cdot C=9$.
\end{itemize}
\end{lemm}

\begin{proof}  The results follow immediately from Lemma~\ref{lem:totalgeod} (here, all the involved curves are irreducible) and the results in \S\ref{sec:self} and \S\ref{sec:config}.

First, note that since the normalizations of the curves in (a) and (b) all have genus 4, their intersection with $K_X$ is always $9$ by Lemma~\ref{lem:totalgeod}. We leave the other computations to the reader and just observe that:

\ni -- a curve $E_i$ can only intersect a curve $E_j$ at $\projmap^{-1}(P_2)$,

\ni -- two local branches of a curve $C$ can only intersect at $\projmap^{-1}(P_2)$,

\ni -- a curve $C$ can only intersect a curve $E_i$ at $\projmap^{-1}(P_2)$ and $\projmap^{-1}(P_3)$.
\end{proof}

From now on, for any two divisors $D$ and $D'$ on $X$, $D\equiv D'$ will mean that $D$ and $D'$ are numerically equivalent.
\begin{lemm}\label{lem:numind}
 $E_1, E_2$ and $C$ represent numerically linearly independent elements in the N\'eron-Severi group, where $C=C_1$ or $C_2$.
\end{lemm}

\begin{proof}  Assume that $E_1, E_2$ and $C$ satisfy numerically an identity
$aE_1+bE_2+cC\equiv0$.
By considering the intersection of the above identity with $E_1, E_2$ and $C$ respectively, we conclude that
\begin{displaymath}
0=5a+13b+11c,\quad
0=13a+5b+7c\quad\text{and}\quad
0=11a+7b-c.
\end{displaymath}
The determinant of this linear system is $1296\neq0$.  Hence $a=b=c=0$.
\end{proof}

\begin{coro}\label{coro:picard}
The Picard
number of $X$ is $3$.
\end{coro}

\begin{proof} It follows from the previous lemma that the Picard number is at least $3$, given by the classes of $E_1, E_2$ and $C$. On the other hand, $h^{1,1}(X)=3$ by Lemma~\ref{lem:numericalinv}.
Since the Picard number is bounded from above by $h^{1,1}$, we conclude that the Picard
number is~3.
\end{proof}

\begin{prop}
The canonical line bundle $K_X$ and $E_3$ give rise to the same class
in the N\'eron-Severi group. Moreover, $K_X\equiv E_3\equiv\frac12E_1+\frac12E_2$.
\end{prop}

\begin{proof} From the discussions in the previous section, we know that $E_1, E_2$ and $C=C_1$ 
form a basis of the N\'eron-Severi group (which is torsion free since $H_1(X,\bZ)=\bZ^2$ is torsion free).

Hence we may write $K_X\equiv aE_1+bE_2+cC$
for some rational numbers $a$, $b$ and $c$.  By pairing with $E_1$, $E_2$ and $C$ respectively, we arrive at
\begin{displaymath}
9=5a+13b+11c,\quad
9=13a+5b+7c\quad\text{and}\quad
9=11a+7b-c.
\end{displaymath}
Solving the above system of equations, we obtain 
$K_X\equiv\frac12E_1+\frac12E_2$.
The same computation leads to $E_3\equiv\frac12(E_1+E_2)$ since $E_3\cdot E_i=K_X\cdot E_i$ for 
$i=1,2$ and $E_3\cdot C=K_X\cdot C$.
\end{proof}

\begin{rem}
By the previous proposition, we also have $K_X\equiv  \frac 23(\frac12E_1+\frac12E_2)+\frac13 E_3=\frac13(E_1+E_2+E_3)$. This fact can be recovered directly from the description of $X$ as an orbifold covering of $R=\bar\Gamma\backslash\twoball$ as in~\S\ref{sec:summary}.

We use the notation of \S\ref{sec:descorb}. Let $\pro:\bar Q=\P^2_\C\to R=\P^2_\C/\Sigma_3$ be the projection.
First, we compute the canonical divisor $K_R$ of $R$. We have $K_R=a D_A=2aD_B$ for some $a\in\Q$ (see~\cite[\S11.4 and Proposition 11.5]{DM2} for a description of ${\rm Pic}(R)$). If we denote by $L=\cO(1)$ the positive generator of ${\rm Pic}(\P^2_\C)$, we have $-3L=K_{\P^2_\C}=\pro^*K_R+3L=6aL+3L$ as $\pro$ branches at order 2 along $D_A$, and $D_A$ has three lines as a preimage in $\P^2_\C$. Hence $K_R=- D_A=-2D_B$.

Now, the orbifold canonical divisor of $\bar\Gamma\backslash\twoball$ is $K_R+\frac{3-1}3D_A+\frac{4-1}4D_B=(-1+\frac23+\frac38)D_A=\frac{1}{24}D_A=\frac1{12}D_B$. In particular, as $\projmap^*D_B=4(E_1+E_2+E_3)$, we get the result.

\end{rem}

\section{Geometry of a generic fiber of the Albanese map}\label{sec:genericalbfiber}

Let $\alpha:X\rightarrow T$ be the Albanese map of $X$. From $\Pi/[\Pi,\Pi]\cong\Z^2$, we know that
$T$ is an elliptic curve, and in particular, $\a$ is onto.  Moreover, note that since the image of  $\a$ is a curve, the fibers of $\a$ are connected (see~\cite[Proposition 9.19]{U}). Let $D$ be a curve on $X$.  The mapping $\alpha$ induces a mapping $\alpha|_D:D\rightarrow T$.
Suppose $F$ is the generic fiber of $\alpha$.  Then the degree of $\alpha|_D$ is given by $D\cdot F$.

 \begin{lemm} \label{lem:numfiber}
Let $m,n,p$ be the degrees of $E_1$, $E_2$, and~$ C=C_1$,  respectively, over the Albanese torus $T$ of $X$.
The generic fiber $F$ of the Albanese fibration of $X$ satisfies
$$
F\equiv\frac1{72}\big((-3m+5n+2p)E_1+(5m-7n+6p)E_2+2(m+3n-4p)C\big).
$$
\end{lemm}

\begin{proof}  From Lemma~\ref{lem:numind}, we may write numerically 
$F\equiv aE_1+bE_2+cC$ for some rational numbers $a,b,c$.
By pairing with $E_1$, $E_2$ and $C$ respectively, we arrive at
\begin{displaymath}
m=5a+13b+11c,\quad
n=13a+5b+7c\quad\text{and}\quad
p=11a+7b-c.
\end{displaymath}
The lemma follows from solving the above system of equations.
\end{proof}

\begin{lemm}\label{lem:degrees}
The degrees of $E_1, E_2, C=C_1$ over the Albanese torus $T$ of $X$ are given by
$$
m=60,\ n=12, \  p=24.
$$
\end{lemm}

\begin{proof} Let $D$ represent one of the curves $E_1, E_2, C$, $\nu:\hat D\rightarrow D$ the normalization of $D$ and $\hat\alpha=\alpha\circ \nu$.  
Let $\omega$ be a positive $(1,1)$ form on $T$.  Then the degree of $D$ over
$T$ is given by
$\deg(D)=\frac{\int_D \alpha^*\omega}{\int_T\omega}$.  The key is to find the
degree from the
information of the explicit curves that we have.  For this purpose, we use
an analogue of the Riemann bilinear relations.  Let $\eta$ be a holomorphic $1$-form on the smooth Riemann surface $\hat D$.  Let $\{u_i, v_i\}$ be a basis of $\pi_1(\hat D)$ as studied in
\S\ref{sec:imhone}.  Then the Riemann bilinear relation (cf.~\cite[p. 231]{GH}) states  that
$$
\int_{\hat D}\sqrt{-1}\eta\wedge\overline{\eta}=\sqrt{-1}\sum_{i=1}^4\Biggl[\int_{u_i}\eta\overline{\int_{v_i}\eta}-\int_{v_i}\eta\overline{\int_{u_i}\eta}\Biggr]
$$
where we use the same notation for an element of $\pi_1(\hat D)$ and its image in $H_1(\hat D,\bZ)$.
Let us write $T=\bC/(\bZ +\bZ \tau)$  where ${\rm Im}\,\tau>0$.  Let $\omega_T=\sqrt{-1}dz\wedge\overline{dz}$ be the standard $(1,1)$ form
on $\bC$ and hence $T$.  The above formula gives
\begin{equation}\label{eq:riembilT}
\int_T\omega_T=\sqrt{-1}(\overline{\tau}-\tau).
\end{equation}
Pulling back to $D$, the above formula gives
\begin{equation}\label{eq:riembilD}
\begin{aligned}
\int_D\alpha^*\omega_T&=\int_{\hat D}\hat\alpha^*\omega_T
=\int_{\hat D}\sqrt{-1}\hat\alpha^*dz\wedge \hat\alpha^*\overline{dz}\\
&=\sqrt{-1}\sum_{i=1}^4\Biggl[\int_{u_i}\hat\alpha^*dz\overline{\int_{v_i}\hat\alpha^*dz}-\int_{v_i}\hat\alpha^*dz\overline{\int_{u_i}\hat\alpha^*dz}\Biggr]\\
&=\sqrt{-1}\sum_{i=1}^4\Biggl[\int_{\hat\alpha_*(u_i)}dz\overline{\int_{\hat\alpha_*(v_i)}dz}-\int_{\hat\alpha_*v_i}dz\overline{\int_{\hat\alpha_*u_i}dz}\Biggr].
\end{aligned}
\end{equation}
In the above, $\hat\alpha_*:H_1(\hat D,\bZ)\rightarrow H_1(T,\bZ)\cong H_1(X,\bZ)\cong \bZ^2$ refers to the map on $1$-cycles induced by $\hat\alpha$.
Hence the right-hand side of the above expression in terms of the notation in \S\ref{sec:imhone} is (up to sign)
\begin{equation}\label{eq:riembilDc}
\sqrt{-1}\sum_{i=1}^4\Biggl[\int_{f(u_i)}dz\overline{\int_{f(v_i)}dz}-\int_{f(v_i)}dz\overline{\int_{f(u_i)}dz}\Biggr]
=\Biggl[\sum_{i=1}^4 \det(f(u_i),f(v_i))\Biggr]\sqrt{-1} (\overline{\tau}-\tau),
\end{equation}
where $\det(f(u_i),f(v_i))$ stands for the determinant of the two by two matrix formed by the two vectors $f(u_i)$ and $f(v_i)$ from the table
in \S\ref{sec:imhone}. Notice that the resulting number will be positive if and only if the orientation on $\hat D$ coming from the choice of $(u_1,v_1,\dots,u_4,v_4)$ as a symplectic basis of $H_1(\hat D,\Z)$, and the orientation on $T$ induced by the choice of the basis $(e_1,e_2)$ of $H_1(T,\bZ)$ are compatible (i.e. both are the same, or the opposite, as the one induced by the respective complex structures).

Substituting into~(\ref{eq:riembilD}) and~(\ref{eq:riembilDc}) the values of $f(u_i)$ and $f(v_i)$ from the table in \S\ref{sec:imhone}, we conclude the values of
$-60, -12, -24$ for the values of $\sum_{i=1}^4 \det(f(u_i),f(v_i))$ in the case of $E_1, E_2$ and $C$ respectively.
We conclude from ~(\ref{eq:riembilT}),~(\ref{eq:riembilD}) and~(\ref{eq:riembilDc}) that the degrees $m,n,p$ are given by $60, 12$ and $24$ respectively (and that the orientation on $\hat D$ and $T$ are not compatible).
\end{proof}

\begin{theo}\label{th:genus}
A fiber of the Albanese map $\alpha:X\rightarrow T$ represents the same numerical class as $-E_1+5E_2$, and the genus of a generic fiber $F$ is~19.
\end{theo}

\begin{proof}  Substituting the values of $m, n, p$ from the previous lemma into Lemma~\ref{lem:numfiber}, we conclude that
$F$ represents the same class as $-E_1+5E_2$ in the N\'eron-Severi group.
Hence 
$$
F\cdot K_X=-E_1\cdot K_X+5E_2\cdot K_X=36.
$$
On the other hand, from the adjunction formula,
$$
2(g-1)=(K_X+F)\cdot F=K_X\cdot F.
$$
Hence $g=19$.
\end{proof}

\section{Geometry of the Albanese fibration}\label{section:geometry}

\subsection{}Let $X_s$ be the fiber of the Albanese fibration $\a$ at $s\in T$. It is 
connected (see~\S\ref{sec:genericalbfiber}). Now $g(X_s)\geqslant 2$, because 
$X$ has negative holomorphic sectional curvature. Although we will not need this in the sequel, 
we observe that the fibration cannot be locally holomorphically trivial.
Otherwise there is a smooth non-trivial family of holomorphic mappings from 
$X_s$ (where $s\in T$ is generic)
to~$X$.  However, a holomorphic map is harmonic with respect to any K\"ahler
metric on $X_s$ and the Poincar\'e metric on $X$.  As the Poincar\'e metric
on $X$ is strictly negative, it follows from uniqueness of harmonic maps
to a negatively curved K\"ahler manifold in its homotopy class that the family is
actually a singleton, a contradiction.  

\subsection{} The result below is just a rewriting of Proposition X.10 in~\cite{Be}. 
As usual, if $D$ is a (not necessarily reduced) curve, we denote by $g(D)$ its arithmetic 
genus (see~\cite[\S II.11]{BHPV}).
\begin{prop}\label{prop:chijump}
Let $X$ (resp. $C$) be a smooth complex surface (resp. curve) and $\pi:X\rightarrow C$ a 
surjective morphism with connected fibers. Let $D=\sum_{i=1}^k m_iD_i$, ($m_i\geqslant 1$) 
be a singular fiber of $\pi$ and let $\Dred=\sum_{i=1}^k D_i$ be the reduced divisor associated to $D$.
%
Let $\nu:\widehat\Dred\rightarrow\Dred$ be the normalization. For any $x$ in the support of $\Dred$, we define $\dtop_x:=\dim_{\bC}(\nu_*\C_{\widehat{\Dred}}/\C_{\Dred})=\sharp \nu^{-1}(x)-1$ the number of (local) irreducible components of $\Dred$ at $x$ minus $1$ and $\dan_x:=\dim_{\bC}(\nu_*\cO_{\widehat{\Dred}}/\cO_{\Dred})$ so that $\mu_x:=2\dan_x-\dtop_x$ is the Milnor number of $\Dred $at $x$.
We also set $\mu=\sum_{x\in \Dred} \mu_x$.
Then, we have
\begin{equation}\label{eq:eulerchar}
e(\Dred)-e(X_s)=\mu+\Bigl(\sum_{i=1}^k(m_i-1)\bigl(2(g(D_i)-1)-D_i^2\bigr)\Bigr)-\bigl(\Dred\bigr)^2.
\end{equation}
\end{prop}
\begin{proof}
From Lemma~VI.5 and the proof of Proposition~X.10 in~\cite{Be}, we immediately get
\begin{equation*}
e(\Dred)=\mu+2\chi(\cO_{\Dred})=\mu+e(X_s)+2(\chi(\cO_{\Dred})-\chi(\cO_D)),
\end{equation*}
where we used the fact that the arithmetic genus of the fibers of a morphism from a surface onto a curve is constant. Now, since $D^2=0$,
\begin{equation*}
\begin{aligned}
2(\chi(\cO_{\Dred})-\chi(\cO_D))&=(K_X+D)\cdot D-(K_X+\Dred)\cdot\Dred\\
&=K_X\cdot(D-\Dred)-\bigl(\Dred\bigr)^2\\
&=\sum_{i=1}^k(m_i-1)(K_X+D_i)\cdot D_i-\sum_{i=1}^k(m_i-1) D_i^2-\bigl(\Dred\bigr)^2.
\end{aligned}
\end{equation*}
That $2\dan_x-\dtop_x$ is the Milnor number of $\Dred$ at $x$ is proved in~\cite[Proposition~1.2.1]{BG}.
\end{proof}

\begin{rem}\label{rem:munodal}
In the notation of Proposition~\ref{prop:chijump}, $\mu_x=0$ if and only if $\Dred$ is smooth at $x$ and if $\mu_x=1$ it is easily seen that the singularity of $\Dred$ at $x$ is nodal (see Lemmas 1.2.1 and 1.2.4 in~\cite{BG} for instance).
\end{rem}

%
%
%

\begin{coro}\label{coro:numsing}
Let $I\subset T$ be the set of singular values of the Albanese fibration $\a$.  Then 
\begin{itemize}
\item[(a)] $\sum_{s_o\in I}\bigl(e(X_{s_o})-e(X_s)\bigr)=3$ where $X_s$ is a generic fiber,
\item[(b)] the cardinality of $I$ is at most $3$,
\item[(c)] $\alpha$ has no multiple fiber, and therefore $(X_{s_0})^{\rm red}$ is singular for at least one $s_0\in I$,
\item[(d)] the total number of singular points in the fibers is at most $3$ and if equality holds, the three singularities are nodal and the fibration is stable. More precisely,
\begin{equation}\label{eq:sumMilnor}
\sum_{s_o\in I}\Bigl(\sum_{x\in X_{s_0}} \mu_x\Bigr)=3.
\end{equation}
\end{itemize}
\end{coro}
\begin{proof}
Note first that there are no rational or elliptic curves in $X$ since the holomorphic sectional curvature of a ball quotient is negative. 

(a) From the standard formula for the Euler number of a holomorphic fibration (see ~\cite[Lemma VI.4]{Be} or~\cite[Proposition III.11.4]{BHPV}), we have
\begin{equation*}
3=e(X)=e(T)\cdot e(X_s)+\sum_{s_o\in I}n_{s_o}=\sum_{s_o\in I}n_{s_o},\end{equation*}
where $n_{s_o}=e(X_{s_o})-e(X_s)$ for $s\in T_o:=T-I$.  Here we used the fact that the Euler characteric of $T$ vanishes.

(b) It is well known (see~\cite[Remark III.11.5]{BHPV}), and it can be easily recovered from Proposition~\ref{prop:chijump}, that $n_{s_o}\geqslant 0$ with equality if and only if $X_{s_o}$ is a multiple fiber with $(X_{s_o})^{\rm red}$ smooth elliptic. But as we noticed above, this is impossible in our case thus $n_{s_o}>0$ for any $s_o\in I$. Since 
$\sum_{s_o\in I}n_{s_o}=3$, we conclude in particular that $|I|\leqslant 3$ (and each $n_{s_o}\leqslant 3)$.

(c) Assume first that a fiber $D$ might be written $D=m\Dred$ with $m\geqslant 2$. Then, by (a) and formula~(\ref{eq:eulerchar}), $3\geqslant e(\Dred)-e(X_s)\geqslant (m-1)\sum_{i=1}^k(g(D_i)-1)$ and the only possibility is that $k=1$, $m=2$ and $g(D_1)=2$. However, by Theorem~\ref{th:genus}, $18=g(D)-1=m(g(D_1)-1)=2$, a contradiction.

Now, assume that $D=\sum_{i=1}^k m_i D_i$ with $k\geqslant 2$, $m_i\geqslant 1$ and $m_1\geqslant 2$. Recall that by Zariski's lemma (see~\cite[Lemma III.8.2]{BHPV}) the self intersection of any effective cycle supported on $\Dred$ must be nonpositive, and it is equal to zero if and only if it is proportional to $D$ (in particular $D_1^2<0$). Therefore by formula~(\ref{eq:eulerchar}), $3\geqslant e(\Dred)-e(X_s)\geqslant 3-\bigl(\Dred\bigr)^2$. But $\bigl(\Dred\bigr)^2=0$ if and only if $D=m\Dred$, a case which has already been ruled out.
%
%

(d) is a consequence of the previous points, equation~(\ref{eq:eulerchar}) and Remark~\ref{rem:munodal}. 
\end{proof}

\begin{lemm}
$\deg(\a_*\omega_{X|T})=1$.
\end{lemm}

\begin{proof}Note that we do not know a priori that the fibration $\a$ is stable. The lemma is a direct consequence of~\cite[Chapter 1]{X}, where it is shown that $\a_*\omega_{X|T}$, the direct image of the relative dualizing sheaf
$\omega_{X|T}$, is locally free of rank $g=g(X_s)$, where $s\in T$ is a generic point (as in the classical case of a stable fibration). As a consequence, this is also the case of $R^1\a_*\cO_X$ which is the dual sheaf of $\a_*\omega_{X|T}$.

Then, using the Leray spectral sequence and the Riemann-Roch formula, we get
$$
\chi(\cO_X)=\chi(\cO_{T})-\chi(R^1\a_*\cO_X)=-\deg(R^1\a_*\cO_X) + (g-1)(g(T)-1)=\deg \a_*\omega_{X|T}
$$
since $\deg \a_*\omega_{X|T}=-\deg(R^1\a_*\cO_X)$ and $g(T)=1$. As $\chi(\cO_X)=1$, the result follows.
\end{proof}

\subsection{}\label{sec:albanesetorus}
Recall from \S\ref{sec:Pi} that
the normalizer $N$ of $\Pi$ in $\oGamma$ is generated by the element $j^4$ of order $3$ and $\Pi$, and
 the automorphism group $\Sigma$ of $X$ is given by the group $N/\Pi$, which
has order $3$. Denote by $\sigma$ the automorphisms of~$\twoball$ and of~$\XX$ induced 
by~$j^4$. If $\xi=(z_1,z_2)\in\twoball$, then $\sigma(\xi)=(\cru z_1,\cru z_2)$ where $\cru=\zeta^4$ is a non trivial cube root of unity.

The Albanese map $\alpha:X\to T=\C/\Lambda$ can be lifted to a holomorphic map $\alpha_0:\twoball\to\C$
so that $\alpha_0(O)=0$ (choosing $\Pi O\in X$ as base point when defining $\alpha$):
\begin{center}
\begin{tikzpicture}[description/.style={fill=white,inner sep=2pt}]
\matrix (m) [matrix of math nodes, row sep=2em,
column sep=2.5em, text height=1.5ex, text depth=0.25ex]
{&&\C\\
\twoball&\XX&T\\};
\path[->,font=\scriptsize]
(m-2-1) edge node[auto] {}      (m-2-2)
(m-2-2) edge node[auto] {$\alpha$}      (m-2-3);
\path[->,font=\scriptsize]
(m-1-3) edge node[auto] {}      (m-2-3);
\path[->,font=\scriptsize,dashed]
(m-2-1) edge node[above] {$\alpha_0$}  (m-1-3);
\end{tikzpicture}
\end{center}
If $\pi\in\Pi$, then $\alpha_0(\pi\xi)-\alpha_0(\xi)\in\Lambda$ is independent of~$\xi\in\twoball$,
and so there is a map $\theta_0:\Pi\to\Lambda$ such that 
$\alpha_0(\pi\xi)=\alpha_0(\xi)+\theta_0(\pi)$ for all $\xi\in\twoball$ and $\pi\in\Pi$.
Since $\theta_0$ is a homomorphism, it factors through our abelianization map $f:\Pi\to\Z^2$, see~\S\ref{sec:Pi}.
So there is a homomorphism $\theta:\Z^2\to\Lambda$ such that
\begin{equation}\label{eq:albaneselift}
\alpha_0(\pi\xi)=\alpha_0(\xi)+\theta(f(\pi))\quad\text{for all}\ \xi\in\twoball\ \text{and}\ \pi\in\Pi.
\end{equation}

By the universal property of the Albanese map, there is an automorphism $\sigma_T:T\to T$
such that the following diagram commutes:
\begin{equation}\label{eq:inducedauto}
\begin{tikzpicture}[description/.style={fill=white,inner sep=2pt},baseline=(current bounding box.base)]
\matrix (m) [matrix of math nodes, row sep=2em,
column sep=2.5em, text height=1.5ex, text depth=0.25ex]
{\XX&T\\
\XX&T\\};
\path[->,font=\scriptsize]
(m-1-1) edge node[auto] {$\alpha$}      (m-1-2);
\path[->,font=\scriptsize]
(m-2-1) edge node[auto] {$\alpha$}      (m-2-2);
\path[->,font=\scriptsize]
(m-1-1) edge node[auto] {$\sigma$}      (m-2-1);
\path[->,font=\scriptsize,dashed]
(m-1-2) edge node[auto] {$\sigma_T$}      (m-2-2);
\end{tikzpicture}
\end{equation}
If the automorphism is trivial, then $\alpha_0(\sigma(\xi))-\alpha_0(\xi)\in\Lambda$ for all $\xi\in\twoball$,
and so is constant. Since $\sigma(O)=O$, $\alpha_0(j^4\xi)=\alpha_0(\xi)$ for all~$\xi$, and this implies that
$\theta(f(j^4\pi j^{-4}))=\theta(f(\pi))$ for all $\pi\in\Pi$. But then \eqref{eq:j4actiononabelianization}
implies that $\theta=0$, because $I-\begin{pmatrix}0&-1\\1&-1\end{pmatrix}$ is non-singular hence
$\Pi\xi\mapsto\alpha_0(\xi)$ is a holomorphic function $\XX\to\C$, and so is constant because $X$ is compact, contradicting surjectivity of~$\alpha$.

As a consequence, $\Sigma$ acts non trivially on $T$ and since $\sigma(O)=O$, the action of $\Sigma$ fixes the point $\a(\Pi O)=0+\Lambda$. From this and $|\Sigma|=3$, it follows immediately that 
the elliptic curve has to be $T=\bC/(\bZ+\cru\bZ)$, and the vertical map $\sigma_T$ on 
the right in~\eqref{eq:inducedauto} is $z+\Lambda\mapsto\cru^i z+\Lambda$ with $i=1$ or $2$. Indeed, the automorphism $\sigma_T$ 
which fixes $0+\Lambda$ is induced by a nontrivial $\C$-linear automorphism of $\C$ preserving $\Lambda$ (see~\cite[Proposition~V.12] {Be} 
for instance). Since it has order 3, it must be multiplication by $\cru^i$, where $i=1$ or $2$. Hence $\Lambda$ contains $1$ and $\cru$ (after renormalization of the lattice).

It follows that there are precisely $3$ fixed points of $\Sigma$ on $T$:  a fundamental domain
of $T$ consists of two equilateral triangles and the fixed points are given by a vertex and the centroid of each of the two 
triangles i.e. are the points $p_\nu=\nu(2+\cru)/3+\Lambda$, $\nu=0,1,-1$ (notice that $(1-\cru)^{-1}=(2+\cru)/3$). In particular, we have proved

\begin{lemm}\label{lem:actiontorus}
The action of $\Sigma$ descends to a non-trivial action of $T$.  There are three fixed points in the action of
$\Sigma$ on $T$.  The elliptic curve $T$ is isomorphic to $\bC/(\bZ+\cru\bZ)$.
\end{lemm}

\subsection{} Our purpose now is to determine the fixed points of $\Aut(X)$. Let $p_\nu=\nu(2+\cru)/3+\Lambda$, $\nu=0,1,-1$ be the fixed points of $\Sigma$ on $T$, as given by Lemma~\ref{lem:actiontorus}.

\begin{prop}\label{prop:fixpoints}
\begin{itemize}
\item[(a)] There are altogether nine fixed points of $\Aut(X)$ on $X$.
\item[(b)] The points $O_1, O_2$ and $O_3$ mentioned in \S\ref{sec:summary} are fixed points of $\Sigma$, all lie in the same fiber $\a^{-1}(p_0)$.
\item[(c)] The other fixed points are 6 of the 288 points lying in $\projmap^{-1}(P_5)$ (see~\S\ref{sec:summary} ).
\item[(d)]  Each of the fibers $\a^{-1}(p_j)$ for $j=1,0,-1$ contains exactly three of the nine fixed points of $\Aut(X)$.
\item[(e)] The fixed points $O_i$, $i=1,2,3$ are of type $\frac13{(1,1)}$, and the other six fixed points are of type $\frac13{(1,2)}$.
\end{itemize}
\end{prop}

\begin{proof} (a) (b) and (c): If $\Pi\xi$ is fixed by~$\sigma$, then $\Pi j^4\xi=\Pi\xi$, and so $\pi j^4\xi=\xi$
for some $\pi\in\Pi$. This implies that $\pi j^4$ has finite order. It cannot be trivial,
since $\Pi$ is torsion free. So there is an element~$t$, in the list of representative
nontrivial elements of finite order in~$\bar\Gamma$ given in 
Proposition~\ref{prop:torsionelts}, or the inverse of one of these,
such that $\pi j^4=gtg^{-1}$ for some $g\in\bar\Gamma$. Thus $gtg^{-1}j^{-4}\in\Pi$.
Since the elements $b^\mu k$, $\mu=0,1,-1$ and $k\in K$, form a set of
coset representatives for $\Pi$ in~$\bar\Gamma$, and since $j^4$
normalizes $\Pi$, we can assume that $g=b^\mu k$ for some $\mu$ and~$k$.

So we search through the finite set of elements $b^\mu k t k^{-1}b^{-\mu}j^{-4}$,
checking which are in~$\Pi$. We find that $b^\mu k t k^{-1}b^{-\mu}j^{-4}\in\Pi$ only happens for $t=j^4$
and $t=buv$. When $t=j^4$, we have $b^\mu k t k^{-1}b^{-\mu}j^{-4}
=b^\mu j^4b^{-\mu}j^{-4}$, independent of~$k$. We find that these three elements are in~$\Pi$, i.e., $b^\mu j^4b^{-\mu}j^{-4}=\pi_\mu$ for some $\pi_\mu\in\Pi$. Explicitly,
\begin{equation}
b^\mu j^4b^{-\mu}j^{-4}=\pi_\mu\in\Pi,\ \text{for}\ \pi_0=1,\ \pi_1=a_2a_1^{-2}a_3^{-3}a_1^{-1}\ \text{and}\ \pi_{-1}=a_2^2a_1a_3a_1^{-1}.
\end{equation}
This means that the three points $\Pi(b^\mu.O)$ are fixed by~$\sigma$.
  
For $t=buv$, we find that $b^\mu k t k^{-1}b^{-\mu}j^{-4}\in\Pi$ for
only 18~pairs $(\mu,k)$. This means that $\sigma$
fixes $\Pi(b^\mu k.Q)$ for these 18 $(\mu,k)$'s. If $(\mu,k)$ satisfies 
$b^\mu k t k^{-1}b^{-\mu}j^{-4}\in\Pi$, then so does $(\mu,kj^4)$, since we 
can write $b^\mu j^4=\pi_\mu j^4b^\mu$ for some $\pi_\mu\in\Pi$, as we have just seen.
Moreover, $\Pi(b^\mu kj^4.Q)=\Pi(b^\mu k.Q)$, since $kj^4=j^4k$ and so
\begin{displaymath}
\Pi(b^\mu kj^4.Q)=\Pi(\pi_\mu j^4b^\mu k.Q)=\Pi(j^4b^\mu k.Q)=\sigma\bigl(\Pi(b^\mu k.Q)\bigr)
=\Pi(b^\mu k.Q).
\end{displaymath}
So we need only consider six of the $(\mu,k)$'s, 
and correspondingly setting 
\begin{displaymath}
\begin{aligned}
h_1&=b^{-1}vu j^3,\\
h_4&=b^{-1}v^2u j^3,\\
\end{aligned}
\quad
\begin{aligned}
h_2&=u^{-1}v j,\\
h_5&=v j^2,\\
\end{aligned}
\quad
\begin{aligned}
h_3&=buv^2j^2,\\
h_6&=bvu^{-1}v,\\
\end{aligned}
\end{displaymath}
we have $h_i(buv)h_i^{-1}j^{-4}=\pi_i'\in\Pi$ for $i=1,\ldots,6$; explicitly,
\begin{displaymath}
\begin{aligned}
\pi_1'&=a_2^2a_1a_3^3,\\
\pi_4'&=a_3^3a_1^2a_3^3,\\
\end{aligned}
\quad
\begin{aligned}
\pi_2'&=j^8a_1j^4,\\
\pi_5'&=j^4a_1^{-1}a_2^{-1}j^8,\\
\end{aligned}
\quad
\begin{aligned}
\pi_3'&=j^8a_1a_2^3j^4a_2a_1a_2^{-2}a_1^{-1}.\\
\pi_6'&= a_2a_1^{-1}.\\
\end{aligned}
\end{displaymath}
These six points $\Pi h_iQ$ are distinct, as we see by checking that $(b^{\mu'}k')(buv)^\epsilon(b^\mu k)^{-1}$ is not in~$\Pi$
for $\epsilon=0,1,2$ when $(\mu',k')$ and~$(\mu,k)$ are distinct pairs in the above list.

Note that (a) corresponds to the case of
Proposition 1.2 (2)(b) in Keum~\cite{K}, 
the latter following from 
Lefschetz fixed point formula and holomorphic Lefschetz fixed point formula.

(d): Writing $\alpha(\Pi\xi)=\alpha_0(\xi)+\Lambda$, as before, where
$\alpha_0(O)=0$, we proved in \S\ref{sec:albanesetorus} that
\begin{displaymath}
\alpha_0(j^4\xi)=\cru^i\alpha_0(\xi)\quad\text{for all}\ \xi\in\twoball
\end{displaymath}
for $i=1$ or $2$. We assume that $i=1$ since the case $i=2$ goes along the same line.
Now $bj^4b^{-1}=\pi_1j^4$ for $\pi_1$ as above, and $f(\pi_1)=(-2,-5)$, so 
\begin{displaymath}
\alpha_0(bj^4b^{-1}\xi)
=\alpha_0(\pi_1j^4\xi)
=\alpha_0(j^4\xi)-2+5\cru
=\cru\alpha_0(\xi)-2+5\cru.
\end{displaymath}
In particular, taking $\xi=bO$, we have $\alpha_0(bO)=\cru\alpha_0(bO)-2+5\cru$,
so that
\begin{displaymath}
\alpha_0(bO)=\frac{2+\cru}{3}(-2+5\cru)
=\cru-3
\in\Lambda.
\end{displaymath}
Hence $\alpha(\Pi b.O)=\alpha(\Pi O)=p_0$.

Similarly, $b^{-1}j^4b=\pi_{-1}j^4$ for $\pi_{-1}$  as above, and $f(\pi_{-1})=(-5,1)$, so that
\begin{displaymath}
\alpha_0(b^{-1}j^4b\xi)
=\alpha_0(\pi_{-1}j^4\xi)
=\alpha_0(j^4\xi)+\theta(f(\pi_{-1}))
=\cru\alpha_0(\xi)-5-\cru.
\end{displaymath}
So taking $\xi=b^{-1}O$, we have
$\alpha_0(b^{-1}O)=\cru\alpha_0(b^{-1}O)-5-\cru$, so that
\begin{displaymath}
\alpha_0(b^{-1}O)=\frac{2+\cru}{3}(-5-\cru)
=-3-2\cru
\in\Lambda.
\end{displaymath}
Hence $\alpha(\Pi b^{-1}.O)=\alpha(\Pi O)=p_0$ too.

Recall now that $h_i(buv)h_i^{-1}j^{-4}=\pi_i'\in\Pi$ for $i=1,\ldots,6$, and
so 
\begin{displaymath}
\alpha_0(h_i(buv)h_i^{-1}\xi)
=\alpha_0(\pi_i'j^4\xi)
=\alpha_0(j^4\xi)+\theta(f(\pi_i'))
=\cru\alpha_0(\xi)+\theta(f(\pi_i')).
\end{displaymath}
In particular, taking $\xi=h_iQ$, we get
$\alpha_0(h_iQ)=\cru\alpha_0(h_i Q)+\theta(f(\pi_i'))$,
so that 
\begin{displaymath}
\alpha_0(h_iQ)=\frac{2+\cru}{3}\theta(f(\pi_i')).
\end{displaymath}
Calculating 
\begin{displaymath}
\begin{aligned}
f(\pi_1')&=(-6,2),\\
f(\pi_4')&=(-4,0),\\
\end{aligned}
\quad
\begin{aligned}
f(\pi_2')&=(-4,1),\\
f(\pi_5')&=(-4,3),\\
\end{aligned}
\quad
\begin{aligned}
f(\pi_3')&=(1,-6),\\
f(\pi_6')&=(-3,-2),\\
\end{aligned}
\end{displaymath}
we have
\begin{displaymath}
\begin{aligned}
\theta(f(\pi_1'))&=-6-2\cru\\
\theta(f(\pi_2'))&=-4-\cru\\
\theta(f(\pi_3'))&=1+6\cru\\
\end{aligned}
\begin{aligned}
&\equiv 1-(1-\cru)\\
&\equiv 1+(1-\cru)\\
&\equiv 1
\end{aligned}
\quad\text{and}\quad
\begin{aligned}
\theta(f(\pi_4'))&=-4\\
\theta(f(\pi_5'))&=-4-3\cru\\
\theta(f(\pi_6'))&=-3+2\cru\\
\end{aligned}
\begin{aligned}
&\equiv -1\\
&\equiv -1\\
&\equiv -1+(1-\cru),
\end{aligned}
\end{displaymath}
where the congruences are modulo~3. Hence
$\alpha(\Pi h_iQ)=\frac{2+\cru}{3}+\Lambda$ for $i=1,2,3$
and
$\alpha(\Pi h_iQ)=-\frac{2+\cru}{3}+\Lambda$ for $i=4,5,6$.

(e) If $\gamma\in\bar\Gamma$, then writing $\gamma.(z_1,z_2)=(w_1,w_2)$, a routine calculation
shows that 
\begin{displaymath}
\begin{pmatrix}
\frac{\partial w_1}{\partial z_1}&\frac{\partial w_2}{\partial z_1}\\
\frac{\partial w_1}{\partial z_2}&\frac{\partial w_2}{\partial z_2}
\end{pmatrix},
\end{displaymath}
evaluated at $\xi=(z_1,z_2)$, equals
\begin{displaymath}
\frac{\zeta^2/(r-1)}{(\gamma_{31}\kappa z_1+\gamma_{32}\kappa z_2+\gamma_{33})^2}\begin{pmatrix}
\kappa z_2\widebar\gamma_{23}+(r-1)\,\widebar\gamma_{22}&-(\kappa z_2\widebar\gamma_{13}+(r-1)\,\widebar\gamma_{12})\\
-(\kappa z_1\widebar\gamma_{23}+(r-1)\,\widebar\gamma_{21})&\kappa z_1\widebar\gamma_{13}+(r-1)\,\widebar\gamma_{11}
\end{pmatrix},
\end{displaymath}
where $\kappa=\sqrt{r-1}$. Taking $\gamma=h_i(buv)h_i^{-1}$ and $\xi=Q=(c_1/\kappa,c_2/\kappa)$
as given in~\eqref{eq:xi3xi8xi12}, we find that this matrix has eigenvalues $\cru^{\pm1}$.
If instead we take $\gamma=b^\mu j^4b^{-\mu}$, and $\xi=b^\mu O$, for $\mu=0,1,-1$, we find that
the matrix is $\cru I$.

Note that (e) is also stated as one of the cases in~\cite[Proposition 1.2]{K},
and was observed by Igor Dolgachev as well.
\end{proof}

\begin{rem}
We do not know whether the fibration $\alpha$ is semistable. 
With some more effort, we can show that the fiber
$\alpha^{-1}(p_0)$ is smooth at each of the three points $O_i$ and that $\alpha$ is not 
semistable if and only if the only singularity of $\alpha$ is a tacnode at one of the 
six other fixed points (see~\cite[Proposition~5]{CKY}).
\end{rem}

\subsection{}\label{subsection:mokquestion}Ngaiming Mok has kindly drawn to our attention the 
following problem which was open and of interest in the geometric study of complex ball quotients.

\begin{ques}
Does there exists a homomorphism $f:X\rightarrow R$
from a smooth complex ball quotient $X$  to a Riemann surface $R$ with a non-totally geodesic singular fiber?
\end{ques}

There are very few explicit examples of mappings from a complex ball quotient to a Riemann surface.  The
known ones described by Deligne-Mostow, Mostow, Livn\'e, Toledo and Deraux all have totally geodesic singular
fibers, cf.~\cite{DM2},~\cite{T} or~\cite{Der2} and the references therein.
We now show that the surface studied in this note provides such an example.

\begin{theo}\label{thm:answertomok}
No singular fiber of the Albanese fibration $\a:X\rightarrow T$ is totally geodesic.
\end{theo}

\begin{proof}  Let $E$ be a singular fiber of $\a$ and let $\hE$ be the normalization of $E$.  Assume for the sake of proof by contradiction that $E$ is totally
geodesic.  According to
Lemma~\ref{lem:totalgeod},
$$
E\cdot E=\frac12 e(\hE)+2\dan(E)
$$
and moreover, $g=g(\hE)+\dan(E)$ and $E\cdot E=0$ since $E$ is a fiber of the fibration, hence $1-g+3\dan(E)=0$. Since we have shown that $g=19$ in Theorem~\ref{th:genus}, this leads to $\dan(E)=6$.  However, totally geodesic curves have simple crossings and computations of Lemma~\ref{lem:totalgeod} and Proposition~\ref{prop:chijump} show that if $P_1,\dots,P_k$ are the singular points of $E$ with $b_i$ local branches at $P_i$ then $\sum_{i=1}^k\mu_{P_i}=\sum_{i=1}^k(2\dan_{P_i}-\dtop_{P_i})=\sum_{i=1}^k(b_i-1)^2\leq 3$ by $(\ref{eq:sumMilnor})$. The only possibility is $k\leq3$ and $b_i=2$ for all $i$  but then $\dan(E)=\frac12\sum_{i=1}^kb_i(b_i-1)\leq 3$, a contradiction.
\end{proof}


\subsection{}In his PhD thesis~\cite{Li}, R.~Livn\'e constructed 
two-ball quotients by taking branched coverings of some generalized universal 
elliptic curves with level structure, and by construction, these surfaces admit
a fibration onto a curve. The Albanese fibration of the Cartwright-Steger
surface does not appear in the same fashion, but one can exhibit another 
(rational) fibration from $X$ onto $\P^1_\C$ appearing in a quite similar way to Livn\'e's. Its generic fiber has genus~109, and
$\sum_{i=1}^4 C_i$, with $C_i$ given in \S\ref{sec:MirrortypeA}, is one of the fibers (cf.~\cite[\S6]{CKY}).


\begin{thebibliography}{12mm}


\bibitem[BHPV]{BHPV} Barth, W. P., Hulek, K., Peters, Chris A. M., Van de
Ven, A., Compact complex surfaces. Second edition. Ergebnisse der
Mathematik und ihrer Grenzgebiete. 3. Folge. A Series of Modern
Surveys in Mathematics 4. Springer-Verlag, Berlin, 2004.

\bibitem[Be]{Be} Beauville, A., Complex algebraic surfaces. Translated from the 1978 French original by R. Barlow, with assistance from N. I. Shepherd-Barron and M. Reid. Second edition. London Mathematical Society Student Texts, 34. Cambridge University Press, Cambridge, 1996.


\bibitem[BG]{BG}  Buchweitz, R.-O., Greuel, G.-M., The Milnor number and deformations of complex curve singularities, Invent. Math. 58 (1980), 241--281.

\bibitem[CKY]{CKY} Cartwright, D., Koziarz, V., Yeung, S.-K., On the Cartwright-Steger surface, long version, arXiv 1412.4137, see also \verb'http://www.maths.usyd.edu.au/u/donaldc/cs-surface/'.

\bibitem[CS1]{CS} Cartwright, D., Steger, T., Enumeration of the $50$ fake projective planes, C. R. Acad. Sci. Paris, Ser. 1,
348 (2010), 11--13, see also weblink provided,\\
\verb'http://www.maths.usyd.edu.au/u/donaldc/fakeprojectiveplanes/'.

\bibitem[CS2]{CS2} Cartwright, D., Steger, T., Finding generators and relations for groups acting on the hyperbolic ball, preprint.

\bibitem[DM1]{DM1}  Deligne, P., Mostow, G. D., Monodromy of hypergeometric functions and nonlattice integral monodromy, Inst. Hautes \'Etudes Sci. Publ. Math. 63 (1986), 5--89.

\bibitem[DM2]{DM2} Deligne, P., Mostow, G. D., Commensurabilities among
 lattices in ${\rm PU}(1,n)$.  Annals of Mathematics Studies, 132.
 Princeton University Press, Princeton, NJ, 1993.

\bibitem[Der1]{Der1}  Deraux, M., A negatively curved K\"ahler threefold not covered by the ball, Invent. Math. 160 (2005), 501--525.

\bibitem[Der2]{Der2}  Deraux, M., Forgetful maps between Deligne-Mostow ball quotients, Geom. Dedicata 150 (2011), 377--389

\bibitem[GH]{GH} Griffiths, P., Harris, J., Principles of algebraic geometry, John Wiley \& Sons, Inc., 1978.


\bibitem[K]{K} Keum, J., Toward a geometric construction of fake projective planes, Rend. Lincei Mat. Appl. 23 (2012), 137--155.

\bibitem[Li]{Li} Livn\'e, R., On certain covers of the universal elliptic curve, Ph. D. Thesis, Harvard University, 1981

\bibitem[Mo1]{Mos1} Mostow, G. D, On a remarkable class of polyhedra in complex hyperbolic space, Pacific J. Math. 86 (1980), 171--276.

\bibitem[Mo2]{Mos2} Mostow, G. D, Generalized Picard lattices arising from half-integral conditions, Inst. Hautes \'Etudes Sci. Publ. Math. 63 (1986), 91--106.

(Fifth Edition), Wiley, 1991.


\bibitem[Pa]{Parker} Parker, J. R., Complex hyperbolic lattices, 
Contemp. Math. 501 (2009), 1--42.

\bibitem[PY]{PY} Prasad, G., Yeung, S-K.,  Fake projective planes, Inv. Math. 168 (2007), 321--370; Addendum, Invent. Math. 182 (2010), 213--227.

\bibitem[R]{R} R\'emy, B., Covolume des groupes $S$-arith\'emiques et faux plans projectifs, [d'apr\`es Mumford, Prasad, Klingler, Yeung, Prasad-Yeung],
S\'eminaire Bourbaki, 60\`eme ann\'ee, 2007-2008, $n^o$ 984.

\bibitem[Sa]{Sa}  Sauter, J. K., Jr, Isomorphisms among monodromy groups and applications to lattices in ${\rm PU}(1,2)$, Pacific J. Math. 146 (1990), 331--384.

\bibitem[ST]{ST}  Shephard, G. C., Todd, J. A., Finite unitary reflection groups, Canadian J. Math. 6 (1954), 274--304.

\bibitem[T]{T} Toledo, D., Maps between complex hyperbolic surfaces, Special volume dedicated to the memory of Hanna Miriam Sandler (1960--1999), Geom. Dedicata 97 (2003), 115--128.

\bibitem[U]{U} Ueno, K., Classification theory of algebraic varieties and compact complex surfaces, Lecture Notes in Mathematics, 439, Springer-Verlag, Berlin, 1975.


\bibitem[X]{X} Xiao, G., Surfaces fibr\'ees en courbes de genre deux, Lecture Notes in Mathematics, 1137, Springer-Verlag, Berlin, 1985.

\bibitem[Y1]{Y1} Yeung, S.-K., Classification of fake projective planes, Handbook of
Geometric Analysis, Vol II, Higher Education Press, Beijing, 2010, 391-431. 

\bibitem[Y2]{Y} Yeung, S.-K., Classification of surfaces of general type with Euler number $3$, Journ. f\"ur die reine und ang. math., 679(2013), 1--22, corrected version,  \verb'http://www.math.purdue.edu/'$\sim$\verb'yeung/'.


\end{thebibliography}
\end{document}